\documentclass[draft,12pt]{article}

\usepackage{amsthm,amsmath,amssymb} 
\usepackage[T1]{fontenc}

\def\Cset{\mathbb{C}}
\def\Rset{\mathbb{R}}
\def\Nset{\mathbb{N}}
\def\Kset{\mathbb{K}}

\theoremstyle{plain}
\newtheorem{thm}{Theorem}[section]
\newtheorem{lem}[thm]{Lemma}
\newtheorem{cor}[thm]{Corollary}
\newtheorem{prop}[thm]{Proposition}

\theoremstyle{definition}
\newtheorem{defn}[thm]{Definition}
\newtheorem{rem}[thm]{Remark}
\newtheorem{exmp}[thm]{Example}

\theoremstyle{remark}
\newtheorem*{prS11}{Proof of (S11)}
\newtheorem*{prS12}{Proof of (S12)}
\newtheorem*{prS13}{Proof of (S13)}

\title{A unified theory of cone metric spaces and its applications to the fixed point theory}
\author{Petko~D.~Proinov}
\date{}

\begin{document}
\maketitle
\vspace{-22pt}
\begin{center}
\emph{Faculty of Mathematics and Informatics, University of Plovdiv, Plovdiv 4000, Bulgaria}\\
\emph{proinov@uni-plovdiv.bg}
\end{center}
\begin{center}
\date{\today}
\end{center}

\medskip

\begin{abstract}
In this paper we develop a unified theory for cone metric spaces over a solid vector space.
As an application of the new theory we present full statements of the iterated contraction principle and the Banach contraction principle in cone metric spaces over a solid vector space.


\end{abstract}

\tableofcontents

\section{Introduction}
\label{sec:Introduction}

In 1905, the famous French mathematician Maurice Fr\'echet \cite{Fre05,Fre06} introduced the 
concept of metric spaces. 
In 1934, his PhD student the Serbian mathematician \DJ{}uro Kurepa \cite{Kur34} 
introduced more abstract metric spaces, in which the metric takes values in an ordered vector space.
In the literature the metric spaces with vector valued metric are known under various names: 
pseudometric spaces \cite{Kur34,Col64}, 
$K$-metric spaces \cite{EL75a,Zab97,RPS06}, 
generalized metric spaces \cite{Rze80},
vector-valued metric spaces \cite{AK11}, 
cone-valued metric spaces \cite{Chu81,Chu82},
cone metric spaces \cite{HZ07,JKR11}.

It is well known that cone metric spaces and cone normed spaces have deep applications in the numerical analysis and the fixed point theory.
Some applications of cone metric spaces can be seen in Collatz \cite{Col64} and Zabrejko \cite{Zab97}.
Schr\"oder \cite{Sch56a,Sch56b} was the first who pointed out the important role of cone metric spaces in the numerical analysis.
The famous Russian mathematician Kantorovich \cite{Kan39} was the first who showed the importance of cone normed spaces for the numerical analysis.

Starting from 2007 many authors have studied cone metric spaces 
over solid Banach spaces and fixed point theorems in such spaces 
(Huang and Zhang \cite{HZ07},
Rezapour and Hamlbarani \cite{RH08},
Wardowski \cite{War09},
Pathak and Shahzad \cite{PS09},
Sahin and Telsi \cite{ST10},
Amini-Harandi and Fakhar \cite{AF10},
S\"onmez \cite{Son10}, 
Latif and Shaddad \cite{LS10},
Turkoglu and Abuloha \cite{TA10},
Khamsi \cite{Kha10},
Radenovi\'c and Kadelburg \cite{RK11},
Khani and Pourmahdian \cite{KP11},
Asadi, Vaezpour and Soleimani \cite{AVS11} and others)

Recently, some authors have studied cone metric spaces over solid topological vector spaces and fixed point theorems in such spaces
(Beg, Azam and Arshad \cite{BAA09}, 
Du \cite{Du10,Du11}, 
Azam, Beg and Arshad \cite{ABA10}, 
Jankovi\'c, Kadelburg and Radenovi\'c \cite{JKR11}, 
Kadelburg, Radenovi\'c and Rako\v cevi\'c \cite{KRR11},
Arandelovi\'c and Ke\v cki\'c \cite{AK11},
Simi\'c \cite{Sim11},
\c Cakalli, S\"onmez and Gen\c c \cite{CSG10} and others)

The purpose of this paper is three-fold. 
First, we develop a unified theory for solid vector spaces. 
Second, we develop a unified theory for cone metric spaces over a solid vector space. 
Third, we present full statements of the iterated contraction principle and the Banach contraction principle in cone metric spaces over a solid vector space. 
The main results of the paper generalize, extend and complement some recent results of
Wei-Shih Du (2010), Kadelburg, Radenovi\'c and Rako\v cevi\'c (2011),
Pathak and Shahzad (2009), Wardowski (2009), Radenovi\'c and Kadelburg (2011) and others.

The paper is structured as follows.

In Section~\ref{sec:VectorSpacesWithConvergence} we introduce a simplified definition of a vector space with convergence which does not require an axiom for the uniqueness of the limit of a convergent sequence. Our axioms are enough to prove some fixed point theorems in cone metric spaces over solid vector spaces.

In Section~\ref{sec:SolidConesInVectorSpacesWithConvergence} we present a criterion for the interior of a solid cone in a vector space with convergence (Theorem~\ref{thm:SolidCone-Criterion}).  

In Section~\ref{sec:OrderedVectorSpaces} we introduce the definition of an ordered vector space and the well known theorem that the vector orderings and cones in a vector space with convergence are in one-to-one correspondence.

In Section~\ref{sec:StrictVectorOrderingsAndSolidCones} we introduce the new notion of a strict vector ordering on an ordered vector space. Then we show that an ordered vector space can be equipped with a strict vector ordering if and only if it is a solid vector space (Theorem~\ref{thm:StrictVectorOrdering-SolidCone}). 
Moreover, if the positive cone of a vector space is solid, then there exists only one strict vector ordering on this space. Hence, the strict vector orderings and solid cones in an vector space with convergence are in one-to-one correspondence.

In Section~\ref{sec:OrderTopologyOnSolidVectorSpaces}, we show that every solid vector space can be endowed with an order topology $\tau$ and that 
${x_n \rightarrow x}$ implies ${x_n \stackrel{\tau}{\rightarrow} x}$ 
(Theorems \ref{thm:OrderTopology1} and \ref{thm:OrderTopology2}). 
As a consequence we show that every convergent sequence in a solid vector space has a unique limit (Theorem~\ref{thm:ConvergenceSolidVectorSpace-UniqueLimit}).

In Section~\ref{sec:MinkowskiFunctionalOnSolidVectorSpaces}, using the Minkowski functional, we show that the order topology on every solid vector space is normable. We also show that every normal and solid vector space $Y$ is normable
in the sense that there exists a norm ${\|\, \, . \, \|}$ on $Y$ such that ${x_n \rightarrow x}$ if and only if 
${x_n \stackrel{\|.\,\|}{\rightarrow} x}$ (Theorem~\ref{thm:MinkowskiFunctional-SolidCone}). 
Also we show that the convergence of sequences in a normal and solid vector space has the properties of the convergence in $\Rset$ (Theorem~\ref{thm:ConvergenceNormalSolidVectorSpace-Properties}).
This result shows that the Sandwich theorem plays an important role in solid vector spaces. 

In Section~\ref{sec:ConeMetricSpacesConeNormedSpaces} we introduce the definitions of cone metric spaces and cone normed spaces. Note that in our definition of a cone normed space $(X, \| . \|)$ we allow $X$ to be a vector space over an arbitrary valued field $\Kset$.

In Section~\ref{sec:ConeMetricSpacesOverSolidVectorSpaces} we study cone metric spaces over solid vector spaces.
The theory of such cone metric spaces is very close to the theory of the usual metric spaces.
For example, every cone metric space over a solid vector space is a metrizable topological space (Theorem~\ref{thm:CMS-Metrizability}) and in such spaces the nested ball theorem holds (Theorem~\ref{thm:NestedBallTheorem}). 
Among the other results in this section we prove that every cone normed space over a solid vector space is normable (Theorem~\ref{thm:CNS-Normability}).
Also in this section we give some useful properties of cone metric spaces which allow us to establish convergence results for Picard iteration with a priori and a posteriori error estimates.
Some of the results in this section generalize, extend and complement some results of
Du \cite{Du10},
Kadelburg, Radenovi\'c and Rako\v cevi\'c \cite{KRR11,KRR09},
\c Cakalli, S\"onmez and Gen\c c \cite{CSG10},
Simi\'c \cite{Sim11},
Abdeljawad and Rezapour \cite{AR11},
Arandelovi\'c and Ke\v cki\'c \cite{AK11},
Amini-Harandi and Fakhar, \cite{AF10},
Khani and Pourmahdian \cite{KP11},
S\"onmez \cite{Son10},
Asadi, Vaezpour and Soleimani \cite{AVS11},
{\c S}ahin and Telsi \cite{ST10}.    
Azam, Beg and Arshad \cite{ABA10}.

In Section~\ref{sec:IteratedContractionsInConeMetricSpaces} we establish a full statement of the iterated contraction principle in cone metric spaces over a solid vector space.
The main result of this section (Theorem~\ref{thm:CMS-IterationContractionPrinciple2}) generalizes, extends and complements some results of Pathak and Shahzad \cite{PS09}, 
Wardowski \cite{War09}, Ortega and Rheinboldt \cite[Theorem {12.3.2}]{OR70} and others.

In Section~\ref{sec:ContractionMappingsInConeMetricSpaces} we establish a full statement of the Banach contraction principle in cone metric spaces over a solid vector space.
The main result of this section (Theorem~\ref{thm:CMS-BanachContractionPrinciple}) 
generalizes, extends and complements 
some results of 
Rezapour and Hamlbarani \cite{RH08},
Du \cite{Du10},
Radenovi\'c and Kadelburg \cite{RK11} and others.

\section{Vector spaces with convergence}
\label{sec:VectorSpacesWithConvergence}

In this section we introduce a simplified definition for the notion of vector spaces with convergence. Our definition is different from those given in the monograph of Collatz \cite{Col64} and in the survey  paper of Zabrejko \cite{Zab97}. In particular, we do not need an axiom for the uniqueness of the limit of a convergence sequence. 

\begin{defn} \label{df:VectorSpaceConvergence}
Let $Y$ be a real vector space 
and let $S$ be the set of all infinite sequences in $Y$.
A binary relation $\rightarrow$ between $S$ and $Y$ 
is called a \emph{convergence} on $Y$ if it satisfies the following axioms:
\begin{description}
	\item (C1) If $x_n \rightarrow x$ and $y_n \rightarrow y$, then 
	$x_n + y_n \rightarrow x+y$.
	\item (C2) If $x_n \rightarrow x$ and $\lambda \in \Rset$, then 
	$\lambda\, x_n \rightarrow \lambda\, x$.
	\item (C3) If $\lambda_n \rightarrow \lambda$ in $\Rset$ and $x \in Y$, then 
	$\lambda_n\, x \rightarrow \lambda\, x$.
\end{description}
The pair ${(Y,\rightarrow)}$ is said to be a \emph{vector space with convergence}.
If ${x_n \rightarrow x}$, then ${(x_n)}$ is said to be a \emph{convergent sequence} in $Y$, and the vector $x$ is said to be a \emph{limit} of ${(x_n)}$. 
\end{defn}

The following two properties of the convergence in a vector space ${(Y,\rightarrow)}$ follow immediately from the above axioms. 
\begin{description}
	\item (C4) If ${x_n = x}$ for all $n$, then $x_n \rightarrow x$.
	\item (C5) The convergence and the limits of a sequence do not depend on the change of finitely many of its terms. 
	\end{description}

\begin{defn} \label{df:OpenClosedSet}
Let ${(Y,\rightarrow)}$ be a vector space with convergence.
\begin{description}
	\item (a) A set ${A \subset Y}$ is said to be (\emph{sequentially}) \emph{open} 
	if ${x_n \rightarrow x}$ and ${x \in A}$ imply 
	${x_n \in A}$ for all but finitely many $n$.
	\item (b) A set ${A \subset Y}$ is said to be (\emph{sequentially}) \emph{closed}
	if ${x_n \rightarrow x}$ and ${x_n \in A}$ for all $n$ imply ${x \in A}$.
\end{description}
\end{defn}

\begin{rem}
Let ${(Y,\rightarrow)}$ be a vector space with convergence.
It is easy to prove that if a set $A \subset Y$ is open, then $Y \backslash A$ is closed.
Let us note that the converse holds true provided that each subsequence of a convergent sequence in $Y$ is convergent with the same limits.
\end{rem}

The following lemma follows immediately from the definition of an open set.

\begin{lem} \label{lem:OpenSets-Topology}
Let ${(Y,\rightarrow)}$ be a vector space with convergence. 
The open sets in ${\,Y}$ satisfies the following properties:
\begin{enumerate}
	\item {\large$\varnothing$} and $Y$ are open.
	\item The union of any family of open sets is open.
	\item The intersection of any finite family of open sets is open.
\end{enumerate}
\end{lem}
 
\begin{rem}
Lemma~\ref{lem:OpenSets-Topology} shows that the family of all open subsets of 
${(Y,\rightarrow)}$ defines a topology on $Y$.
Note that in this paper we will never consider this topology on $Y$. 
\end{rem}

\begin{lem} \label{lem:OpenSets-Properties}
Let ${(Y,\rightarrow)}$ be a vector space with convergence.
Suppose $U$ and $V$ are nonempty subsets of $Y$.
Then the following statements hold true.
\begin{enumerate}
	\item If $U$ is open and $\lambda > 0$, then $\lambda U$ is open.
	\item If $U$ or $V$ is open, then $U + V$ is open.
\end{enumerate}
\end{lem}

\begin{proof}
(i) Let ${\lambda > 0}$ and $U$ be an open subset of $Y$. Suppose ${(x_n)}$ is a convergent sequence in $Y$ with a limit ${x \in \lambda U}$. Then there exists a vector ${a \in U}$ such that ${x = \lambda a}$. Consider the sequence ${(a_n)}$ defined by 
${a_n = \frac{1}{\lambda}\, x_n}$. It follows from (C2) that ${a_n \rightarrow a}$ 
since ${a = \frac{1}{\lambda}\, x}$.
Taking into account that $U$ is open and ${a \in U}$, we conclude that 
${a_n \in U}$ for all but finitely many $n$. Then ${x_n \in \lambda U}$ for the same $n$ since 
${x_n = \lambda a_n}$. Therefore, the set ${\lambda U}$ is open.

(ii) Let $U$ be an arbitrary subset of $Y$ and $V$ be an open subset of $Y$.
Suppose ${(x_n)}$ is a convergent sequence in $Y$ with limit ${x \in U + V}$. Then there exist ${a \in U}$ and ${b \in V}$ such that ${x = a + b}$. Consider the sequence ${(b_n)}$ defined by ${b_n = x_n - a}$. It follows from (C1) and (C4) that ${b_n \rightarrow b}$  since ${b = x - a}$. Taking into account that $V$ is open and ${b \in V}$, we conclude that ${b_n \in V}$ for all but finitely many $n$. Then ${x_n \in U + V}$ for these $n$ since 
${x_n = a + b_n}$. Therefore, the set ${U + V}$ is open.
\end{proof}

Due to the first two statements of Lemma~\ref{lem:OpenSets-Topology} we can give the following definition. 

\begin{defn} \label{df:Interior}
Let $A$ be a subset of a vector space ${(Y,\rightarrow)}$.
The \emph{interior} ${A^{\circ}}$ of $A$ is called the biggest open subset contained in $A$, that is, ${A^{\circ} = \bigcup{U}}$ where $\bigcup$ ranges through the family of all open subsets of $Y$ contained in $A$.    
\end{defn}

The following lemma follows immediately from the definition of the notion of interior.

\begin{lem} \label{lem:Interior-Property}
Let $A$ and $B$ be two subsets of a vector space ${(Y,\rightarrow)}$. 
Then 
\[
A \subset B \text { implies } A^{\circ} \subset B^{\circ}.
\]
\end{lem}

\begin{exmp}
Let ${(Y, \tau)}$ be an arbitrary topological vector space and let ${\stackrel{\tau}{\rightarrow}}$ be the $\tau$-convergence in $Y$. 
Obviously, ${(Y, \stackrel{\tau}{\rightarrow})}$ is a vector space with convergence. 
It is well known that every $\tau$-open subset of ${(Y, \tau)}$ is sequentially open and every $\tau$-closed set is sequentially closed.  
Recall also that a topological space is called a \emph{sequential space} if it satisfies one of the following equivalent conditions:
\begin{description}
	\item (a) Every sequentially open subset of $Y$ is $\tau$-open.
	\item (b) Every sequentially closed subset of $Y$ is $\tau$-closed.
\end{description}
Let us note that according to a well known theorem of Franklin \cite{Fra65} every first countable topological vector space is a sequential space. 
For sequential topological spaces see a survey paper of Goreham \cite{Gor04}.
\end{exmp}

\section{Solid cones in vector spaces with convergence}
\label{sec:SolidConesInVectorSpacesWithConvergence}

In this section we establish a useful criterion for the interior of a solid cone.
This criterion will play an important role in 
Section~\ref{sec:StrictVectorOrderingsAndSolidCones}.

For more on cone theory, see the classical survey paper of 
Krein and Rutman \cite{KR48}, 
the classical monographs of 
Krasnoselskii \cite[Chapter~1]{Kra62}, 
Deimling \cite[Chapter~6]{Dei85}, 
Zeidler \cite[Section 1.6]{Zei95} 
as well as the recent monograph of Aliprantis and Tourky \cite{AT07}.
 
\begin{defn} \label{df:Cone}
A nonempty closed subset $K$ of a vector space ${(Y,\rightarrow)}$ is called a \emph{cone} if it satisfies the following properties:
\begin{enumerate}
	\item $\lambda K \subset K$ for any $\lambda \ge 0$;
	\item $K + K \subset K$;
	\item $K \cap (-K) = \{ 0 \}$.
\end{enumerate}
A cone $K$ is called \emph{trivial} if ${K =\{0\}}$.
A nontrivial cone $K$ is said to be a \emph{solid cone} if its interior is nonempty.
\end{defn}

\begin{lem} \label{lem:SolidCone-Criterion}
Let $K$ be cone in a vector space ${(Y,\rightarrow)}$. 
Then there is at most one nonempty open subset $U$ of $K$ satisfying the following conditions:
\begin{enumerate}
	\item $\lambda U \subset U$ for any $\lambda > 0$;
	\item $K + U \subset U$;
	\item $\, 0 \notin U$.
\end{enumerate}
\end{lem}

\begin{proof}
Let $U$ be a nonempty open subsets of $K$ satisfying conditions (i)-(iii). First we shall prove that every nonempty open subset $V$ of $K$ is a subset of $U$.
Let a vector ${x \in V}$ be fixed.
Choose a vector ${a \in U}$ with ${a \neq 0}$. This is possible since $U$ is nonempty and 
${0 \notin U}$. Consider the sequence ${(x_n)}$ in $Y$ defined by 
${x_n = x - \frac{1}{n}\, a}$.    
It follows from (C1) and (C4) that ${x_n \rightarrow x}$. 
Since $V$ is open and ${x \in V}$, then there exists ${n \in \Nset}$ 
such that ${x_n \in V}$. Therefore, ${x_n \in K}$ since ${V \subset K}$. 
On the other hand it follows from (i) that ${\frac{1}{n}\, a \in U}$ since ${a \in U}$.
Then from (ii) we conclude that ${x = x_n + \frac{1}{n}\, a \in U}$
which proves that ${V \subset U}$.
Now if $U$ and $V$ are two nonempty open subsets of $K$ satisfying conditions (i)-(iii),
then we have both ${V \subset U}$ and ${U \subset V}$ which means that ${U=V}$.
\end{proof}

Now we are ready to establish a criterion for the interior of a solid cone.

\begin{thm} \label{thm:SolidCone-Criterion}
Let $K$ be a solid cone in a vector space ${(Y,\rightarrow)}$. 
Then the interior ${K^{\circ}}$ of $K$ has the following properties:
\begin{enumerate}
	\item $\lambda K^{\circ} \subset K^{\circ}$ for any $\lambda > 0$;
	\item $K + K^{\circ} \subset K^{\circ}$;
	\item $\, 0 \notin K^{\circ}$.
\end{enumerate}
Conversely, if a nonempty open subset ${K^{\circ}}$ of $K$ satisfies properties \emph{(}i\emph{)}-\emph{(}iii\emph{)}, then ${K^{\circ}}$ is just the interior of $K$.
\end{thm}

\begin{proof}
\emph{First part}.
We shall prove that the interior ${K^{\circ}}$ of a solid cone $K$ satisfies properties (i)-(iii).

(i) Let ${\lambda > 0}$. It follows from Lemma~\ref{lem:OpenSets-Properties}(i) that 
${\lambda K^{\circ}}$ is open.  From ${K^{\circ} \subset K}$ and ${\lambda K \subset K}$, we obtain ${\lambda K^{\circ} \subset K}$. This inclusion and Lemma~\ref{lem:Interior-Property} imply ${\lambda K^{\circ} \subset K^{\circ}}$.

(ii) By Lemma~\ref{lem:OpenSets-Properties}(ii) ${K + K^{\circ}}$ is an open set.  It follows from ${K^{\circ} \subset K}$ and ${K + K \subset K}$ that 
${K + K^{\circ} \subset K}$. Now from Lemma~\ref{lem:Interior-Property}, we conclude that 
${K + K^{\circ} \subset K^{\circ}}$.

(iii) Assume that ${0 \in K^{\circ}}$.
Since $K$ is nonempty and nontrivial, then we can choose a vector ${a \in K}$ with 
${a \ne 0}$. By axiom (C3), ${-\frac{1}{n}\, a \rightarrow 0}$. Taking into account that ${K^{\circ}}$ is open, we conclude that there exists ${n \in \Nset}$ such that 
${-\frac{1}{n} a \in K^{\circ}}$. Then it follows from (i) that ${-a \in K^{\circ}}$. Since ${K^{\circ} \subset K}$, we have both ${a \in K}$ and ${-a \in K}$ which implies 
${a = 0}$.
This is a contradiction which proves that ${\, 0 \notin K^{\circ}}$.

\emph{Second part}.
The second part of the theorem follows from Lemma~\ref{lem:SolidCone-Criterion}.
Indeed, suppose that ${K^{\circ}}$ is a nonempty open subset of $K$ satisfying properties (i)-(iii). Then by Lemma~\ref{lem:SolidCone-Criterion} we conclude that ${K^{\circ}}$ is a unique nonempty open subset of $K$ satisfying these properties. On the other hand, it follows from the first part of the theorem that the interior of $K$ also satisfies properties (i)-(iii). Therefore, $K^{\circ}$ coincide with the interior of $K$.
\end{proof}

\section{Ordered Vector Spaces}
\label{sec:OrderedVectorSpaces}

Recall that a binary relation $\preceq$ on a set $Y$ is said to be
an \emph{ordering} on $Y$ if it is reflexive, antisymmetric and transitive. 

\begin{defn} \label{df:OrderedVectorSpace}
An ordering $\preceq$ on a vector space with convergence $(Y,\rightarrow)$ 
is said to be a 
\emph{vector ordering} if it is compatible with the algebraic and convergence structures on $Y$ in the sense that the following are true:
\begin{description}
	\item (V1) If $x \preceq y$, then $x+z \preceq y+z$;
	\item (V2) If $\lambda \ge 0$ and $x \preceq y$, then $\lambda x \preceq \lambda y$;
	\item (V3) If $x_n \rightarrow x$, $y_n \rightarrow y$, $x_n \preceq y_n$ for all $n$, then $x \preceq y$.    
\end{description}
	A vector space ${(Y,\rightarrow)}$ equipped with a vector ordering $\preceq$ is called an \emph{ordered vector space} and is denoted by ${(Y,\preceq,\rightarrow)}$.
	If the convergence $\rightarrow$ on $Y$ is produced by a vector topology $\tau$, we sometimes write  ${(Y,\tau,\preceq)}$ instead of ${(Y,\preceq,\rightarrow)}$. Analogously, if the convergence $\rightarrow$ on $Y$ is produced by a norm ${\|\, \, . \, \|}$, we sometimes write ${(Y,\|\, \, . \, \|,\preceq)}$.
\end{defn}

Axiom (V3) is known as \emph{passage to the limit in inequalities}.
Obviously, it is equivalent to the following statement:
\begin{description}
	\item {(V$3^\prime$)} If $x_n \rightarrow 0$, $x_n \succeq 0$ for all $n$, 
	then $x \succeq 0$. 
\end{description}
Every vector ordering $\preceq$ on an ordered vector space 
$(Y,\preceq,\rightarrow)$ satisfies also the following properties:
\begin{description}
	\item (V4) If $\lambda \le 0$ and $x \preceq y$, then $\lambda x \succeq \lambda y$;
	\item (V5) If $\lambda \le \mu$ and $x \succeq 0$, then $\lambda x \preceq \mu x$;
	\item (V6) If $\lambda \le \mu$ and $x \preceq 0$, then $\lambda x \succeq \mu x$;
	\item (V7) If $x \preceq y$  and $u \preceq v$ , then $x+u \preceq y+v$.
\end{description}

\begin{defn} \label{df:Positive Cone} 
Let ${(Y,\preceq,\rightarrow)}$ be an ordered vector space. 
The set
\begin{equation} \label{eq:PositiveCone}
Y_+ = \{ x \in Y : x \succeq 0 \}
\end{equation}
is called the \emph{positive cone} of the ordering $\preceq$ 
or \emph{positive cone} of $Y$.
\end{defn}

The following well known theorem shows that the positive cone is indeed a cone.
It shows also that the vector orderings and cones in a vector space 
${(Y,\rightarrow)}$ with convergence are in one-to-one correspondence.

\begin{thm} \label{thm:VectorOrdering-Cone}
Let ${(Y,\rightarrow)}$ be a vector space with convergence.
If a relation $\preceq$ is a vector ordering on $Y$, then its positive cone is a cone 
in $Y$.
Conversely, if a subset $K$ of $Y$ is a cone, then the relation $\preceq$ on $Y$ defined by means of
\begin{equation} \label{eq:Cone-VectorOrdering}
x \preceq y \quad \text{if and only if} \quad y-x \in K
\end{equation}
is a vector ordering on $Y$ whose positive cone coincides with $K$.
\end{thm}

\begin{defn} \label{df:Bounded/Increasing}
Let ${(Y,\preceq,\rightarrow)}$ be an ordered vector space.
\begin{description}
	\item (a) A set ${A \subset Y}$ is called \emph{bounded} if there exist two vectors in ${a,b \in Y}$ such that
	${a \preceq x \preceq b}$ for all ${x \in A}$. 
	\item (b) A sequence ${(x_n)}$ in $Y$ is called \emph{bounded} if the set of its terms is bounded.
	\item (c) A sequence ${(x_n)}$ in $Y$ is called \emph{increasing} if ${x_1 \preceq x_2 \preceq \ldots}$
	\item (d) A sequence ${(x_n)}$ in $Y$ is called \emph{decreasing} if ${x_1 \succeq x_2 \succeq \ldots}$.	
\end{description}
\end{defn}

\begin{defn} \label{df:SolidSpace}
An ordered vector space ${(Y,\preceq,\rightarrow)}$ is called a 
\emph{solid vector space} if its positive cone is solid.
\end{defn}

\begin{defn} \label{df:NormalSpace}
An ordered vector space ${(Y,\preceq,\rightarrow)}$ is called a 
\emph{normal vector space} whenever for arbitrary sequences ${(x_n)}$, ${(y_n)}$, ${(z_n)}$ in $Y$,
\begin{equation} \label{eq:SandwichTheorem1}
x_n \preceq y_n \preceq z_n \text{  for all  } n \, 
\text{ and }\, x_n \rightarrow x \,\text{ and }\, z_n \rightarrow x  
\quad \text{imply} \quad y_n \rightarrow x. 
\end{equation}
The statement \eqref{eq:SandwichTheorem1} is known as \emph{sandwich theorem} or \emph{rule of intermediate sequence}. 
\end{defn}

\begin{defn} \label{df:RegularSpace}
An ordered vector space ${(Y,\preceq,\rightarrow)}$ is called a \emph{regular vector space} if it satisfies one of the following equivalent conditions.
\begin{description}
	\item (a) Every bounded increasing sequence in $Y$ is convergent.
	\item (b) Every bounded decreasing sequence in $Y$ is convergent.
\end{description}
\end{defn}

\section{Strict vector orderings and solid cones}
\label{sec:StrictVectorOrderingsAndSolidCones}

In this section we introduce a notion of a strict vector ordering and prove that an 
ordered vector space can be equipped with a strict vector ordering if and only if it is a solid vector space. 

Recall that a nonempty binary relation $\prec$ on a set $Y$ is said to be a
\emph{strict ordering} on $Y$ if it is irreflexive, asymmetric and transitive.

\begin{defn} \label{df:StrictVectorOrdering}
Let ${(Y,\preceq,\rightarrow)}$ be an ordered vector space.
A strict ordering $\prec$ on $Y$ is said to be a \emph{strict vector ordering} if it is compatible with the vector ordering, the algebraic structure and  the convergence structure on $Y$ in the sense that the following are true:
\begin{description}
	\item (S1) If $x \prec y$ , then $x \preceq y$;
	\item (S2) If $x \preceq y$ and $y \prec z$, then $x \prec z$;
	\item (S3) If $x \prec y$ , then $x+z \prec y+z$;
	\item (S4) If $\lambda > 0$ and $x \prec y$, then $\lambda x \prec \lambda y$;
	\item (S5) If $x_n \rightarrow x$, $y_n \rightarrow y$ and $x \prec y$, 
	then $x_n \prec y_n$ for all but finitely many $n$.
\end{description}

An ordered vector space ${(Y,\preceq,\rightarrow)}$ equipped with a strict vector ordering
$\prec$ is denoted by ${(Y,\preceq,\prec,\rightarrow)}$. It turns out that ordered vector spaces with strict vector ordering are just solid vector spaces (see Corollary~\ref{cor:SolidVectorSpace=StrictVectorOrdering} below).  
\end{defn}

Axiom (S5) is known as \emph{converse property of passage to the limit in inequalities}.
It is equivalent to the following statement:
\begin{description}
	\item {(S$5^\prime$)} If $x_n \rightarrow 0$ and $c \succ 0$, 
	then $x_n \prec c$ for all but finitely many $n$. 
\end{description}
Strict vector ordering $\prec$ on a ordered vector space 
${(Y,\preceq,\rightarrow)}$ satisfies also the following properties:
\begin{description}
	\item (S6) If $\lambda < 0$ and $x \prec y$, then $\lambda x \succ \lambda y$.
	\item (S7) If $\lambda < \mu$ and $x \succ 0$, then $\lambda x \prec \mu x$.
	\item (S8) If $\lambda < \mu$ and $x \prec 0$, then $\lambda x \succ \mu x$.
	\item (S9) If $x \prec y$ and $y \preceq z$, then $x \prec z$.
	\item (S10) If $x \preceq y$  and $u \prec v$ , then $x+u \prec y+v$;
	\item (S11) If $x \prec c$ for each $c \succ 0$, then $x \preceq 0$.
	\item (S12) For every finite set ${A \subset Y}$ consisting of strictly positive vectors, there exists a vector 
	${c \succ 0}$ such that ${c \prec x}$ for all ${x \in A}$. Moreover, for every vector $b \succ 0$, $c$ always can be chosen in the form ${c = \lambda \, b}$ for some ${\lambda > 0}$.
	\item (S13) For every finite set ${A \subset Y}$, there is a vector ${c \succ 0}$ such that ${-c \prec x \prec c}$ for all ${x \in A}$. Moreover, for every vector $b \succ 0$, $c$ always can be chosen in the form ${c = \lambda \, b}$ for some ${\lambda > 0}$.
	\item (S14) For every ${x \in Y}$ and every ${b \in Y}$ with ${b \succ 0}$, there exists ${\lambda > 0}$ such that 
	${-\lambda \, b \prec x \prec \lambda \, b}$.
	\end{description}

The proofs of properties (S6)--(S10) are trivial. Property (S14) is a special case of (S13). So we shall prove (S11)--(S13). 

\begin{prS11}
Let $x$ be a vector in $Y$ such that $x \prec c$ for each $c \succ 0$. 
Choose a vector ${b \in Y}$ with ${b \succ 0}$. 
It follows from (S4) that ${\frac{1}{n}\, b \succ 0}$ for each ${n \in \Nset}$. 
Hence, ${x \prec \frac{1}{n}\, b}$ for each ${n \in \Nset}$. 
Passing to the limit in this inequality, we obtain ${x \preceq 0}$.
\qed
\end{prS11}

\begin{prS12}
Let $x$ be an arbitrary vector from $A$.
Choose a vector ${b \in Y}$ with ${b \succ 0}$. 
Since ${\frac{1}{n}\, b \rightarrow 0}$ and ${0 \prec x}$, then from (S5) we deduce that 
${\frac{1}{n}\, b \prec x}$ for all but finitely many $n$. 
Taking into account that $A$ is a finite set, we conclude that for sufficiently large $n$ we have ${\frac{1}{n}\, b \prec x}$ for all ${x \in A}$. 
Now every vector ${c = \frac{1}{n}\, b}$ with sufficiently large $n$ satisfies 
${c \prec x}$ for all ${x \in A}$. To complete the proof put ${\lambda = \frac{1}{n}}$.
\qed
\end{prS12}

\begin{prS13}
Let $x$ be an arbitrary vector from $A$.
Choose a vector ${b \in Y}$ with ${b \succ 0}$. 
Since ${\frac{1}{n}\, x \rightarrow 0}$ and ${-\frac{1}{n}\, x \rightarrow 0}$, then from (S5) we obtain that ${\frac{1}{n}\, x \prec b}$ and ${-\frac{1}{n}\, x \prec b}$ for all but finitely many $n$. From these inequalities, we conclude that ${-n b \prec x \prec n b}$.   
Taking into account that $A$ is a finite set, we get that every vector ${c = n b}$ with sufficiently large $n$ satisfies ${-c \prec x \prec c}$ for all ${x \in A}$. To complete the proof put ${\lambda = n}$.
\qed
\end{prS13}

The next theorem shows that an ordered vector space can be equipped with a strict vector ordering if and only if it is a solid vector space. 
Moreover, on every ordered vector space there is at most one strict vector ordering.
In other words the solid cones and strict vector orderings on a vector space with convergence are in one-to-one correspondence.

\begin{thm} \label{thm:StrictVectorOrdering-SolidCone}
Let ${(Y,\preceq,\rightarrow)}$ be an ordered vector space and let $K$ be its positive cone, i.e. ${K = \{ x \in Y : x \succeq 0 \}}$.
If a relation $\prec$ is a strict vector ordering on $Y$, then $K$ is a solid cone with the interior
\begin{equation} \label{eq:interior}
K^{\circ} = \{ x \in Y : x \succ 0 \}.
\end{equation}
Conversely, if $K$ is a solid cone with the interior $K^{\circ}$, then the relation $\prec$ on $Y$ defined by means of 
\begin{equation} \label{eq:Interior-StrictVectorOrdering}
x \prec y \quad \text{if and only if} \quad y-x \in K^{\circ}.
\end{equation}
is a unique strict vector ordering on $Y$.
\end{thm}

\begin{proof}
\emph{First part}.
Suppose a relation $\prec$ is a strict vector ordering on $Y$.
We shall prove that the set ${K^{\circ}}$ defined by \eqref{eq:interior} is a nonempty open subset of $K$ which satisfies conditions (i)-(iii) of Theorem~\ref{thm:SolidCone-Criterion}. Then it follows from the second part of Theorem 3.3 that ${K^{\circ}}$ is the interior of $K$ and that $K$ is a solid cone. By the definition of strict ordering, it follows that the relation $\prec$ is nonempty. Therefore, there are at least two vectors $a$ and $b$ in $Y$ such that ${a \prec b}$. From (S3), we obtain ${b-a \succ 0}$. 
Therefore, ${b-a \in K^{\circ}}$ which proves that ${K^{\circ}}$ is nonempty. 
Now let ${x_n \rightarrow x}$ and ${x \in K^{\circ}}$. By the definition of ${K^{\circ}}$, we get ${x \succ 0}$. Then by (S5), we conclude that ${x_n \succ 0}$ for all but finitely many $n$ which means that ${K^{\circ}}$ is open. Conditions (i) and (ii) of
Theorem~\ref{thm:SolidCone-Criterion} follow immediately from (S4) and (S10) respectively.
It remains to prove that ${\, 0 \notin K^{\circ}}$. 
Assume the contrary, that is ${\, 0 \in K^{\circ}}$. By the definition of ${K^{\circ}}$, we get $0 \succ 0$ which is a contradiction since the relation $\prec$ is irreflexive.    

\emph{Second part}.
Let $K$ be a solid cone and ${K^{\circ}}$ be its interior. Note that according to to the first part of Theorem~\ref{thm:SolidCone-Criterion}, ${K^{\circ}}$ has the following properties: 	${\lambda K^{\circ} \subset K^{\circ}}$ for any ${\lambda > 0}$,
${K + K^{\circ} \subset K^{\circ}}$ and ${0 \notin K^{\circ}}$.
  We have to prove that the relation $\prec$ defined by \eqref{eq:Interior-StrictVectorOrdering} is a strict vector ordering. First we shall show that $\prec$ is nonempty and irreflexive. Since $K$ is solid, ${K^{\circ}}$ is nonempty and nontrivial. Hence, there exists a vector 
${c \in K^{\circ}}$ such that ${c \neq 0}$. Now by the definition of $\prec$, we get 
${0 \prec c}$ which means that $\prec$ is nonempty. To prove that $\prec$ is irreflexive assume the contrary. Then there exists a vector ${x \in Y}$ such that ${x \prec x}$. 
Hence, ${0 = x-x \in K^{\circ}}$ which is a contradiction since ${0 \notin K^{\circ}}$. 
Now we shall show that $\prec$ satisfies properties (S1)--(S5).

(S1) Let $x{ \prec y}$. Using the definition \eqref{eq:Interior-StrictVectorOrdering}, the inclusion ${K^{\circ} \subset K}$, and the definition of the positive cone $K$, we have
\[
x \prec y \, \Rightarrow \, 
y-x \in K^{\circ} \, \Rightarrow \, 
y-x \in K \, \Rightarrow \, 
y-x \succeq 0 \, \Rightarrow \, x \preceq y .  
\]

(S2) Let ${x \preceq y}$ and ${y \prec z}$. Using the definition of the positive cone $K$, the definition \eqref{eq:Interior-StrictVectorOrdering} and the inclusion 
${K + K^{\circ} \subset K^{\circ}}$, we get  
\[
x \preceq y \text{ and } y \prec z \, \Rightarrow \,
y-x \in K \text{ and } z-y \in K^{\circ} \, \Rightarrow \, 
z-x \in K^{\circ} \, \Rightarrow \, x \prec z . 
\]

(S3) follows immediately from the definition \eqref{eq:Interior-StrictVectorOrdering}.
 
(S4) Let ${x \preceq y}$ and ${\lambda > 0}$. 
Using the definition \eqref{eq:Interior-StrictVectorOrdering} and the inclusion 
${\lambda K^{\circ} \subset K^{\circ}}$, we obtain 
\[
x \prec y \, \Rightarrow \, 
y-x \in K^{\circ} \, \Rightarrow \,
\lambda (y-x) \in K^{\circ} \, \Rightarrow \, 
\lambda y - \lambda x \in K^{\circ} \, \Rightarrow \,
\lambda x \prec \lambda y .
\]

(S5) Let ${x_n \rightarrow x}$, ${y_n \rightarrow y}$ and ${x \prec y}$. This yields ${y_n - x_n \rightarrow y - x}$
and ${y-x \in K^{\circ}}$. Since ${K^{\circ}}$ is open, we conclude that 
${y_n - x_n \in K^{\circ}}$ for all but finitely many $n$. 
Hence, ${x_n \prec y_n}$ for all but finitely many $n$.

\emph{Uniqueness}.
Now we shall prove the uniqueness of the strict vector ordering on $Y$. Assume that $\prec$ and $<$ are 
two vector orderings on $Y$. It follows from the first part of the theorem
that 
\[
K^{\circ} = \{ x \in Y : x \succ 0 \} = \{ x \in Y : x > 0 \}.
\]
From this and (S3), we get for all ${x,y \in Y}$,
\[
x \prec y \, \Leftrightarrow \, 
y-x \succ 0 \, \Leftrightarrow \,
y-x \in K^{\circ} \, \Leftrightarrow \,
y-x > 0 \, \Leftrightarrow \,
x < y
\]
which means that relations $\prec$ and $<$ are equal. 
\end{proof}

Note that property (S13) shows that every finite set in a solid vector space is bounded. Property (S12) shows that every finite set consisting of strictly positive vectors in a solid vector space is bounded below by a positive vector. 

The following assertion is an immediate consequence of Theorem~\ref{thm:StrictVectorOrdering-SolidCone}.

\begin{cor} \label{cor:SolidVectorSpace=StrictVectorOrdering}
Let ${(Y,\preceq,\rightarrow)}$ be an ordered vector space. Then the following statements are equivalent.
\begin{enumerate}
	\item $Y$ is a solid vector space.
	\item $Y$ can be equipped with a strict vector ordering.
\end{enumerate}
\end{cor}

\begin{rem}
The strict ordering $\prec$ defined by \eqref{eq:Interior-StrictVectorOrdering} 
was first introduced in 1948 by Krein and Rutman \cite[p.~8]{KR48} in the case when $K$ is a solid cone in a Banach space $Y$. In this case they proved that $\prec$ satisfies axioms {(S1)--(S4)}.
\end{rem}

In conclusion of this section we present three examples of solid vector spaces 
We end the section with a remark which shows that axiom (S5) plays an important role in the definition of strict vector ordering.

\begin{exmp} \label{exmp:R^n}
Let ${Y = \Rset^n}$ with $\rightarrow$ the coordinate-wise convergence, 
and with coordinate-wise ordering defined by
\begin{eqnarray*}
& x \preceq y & \, \text{if and only if} \quad x_i \le y_i \text{ for each } 
i = 1,\dots,n ,\\
& x \prec y & \, \text{if and only if} \quad x_i < y_i \text{ for each } 
i = 1,\dots,n .
\end{eqnarray*}
Then ${(Y,\preceq,\prec,\rightarrow)}$ is a solid vector space.
This space is normal and regular.
\end{exmp}

\begin{exmp} \label{exmp:C[0,1]}
Let ${Y = C[0,1]}$ with the max-norm ${\|\, . \,\|_{\infty}}$.
Define the pointwise ordering $\preceq$ and $\prec$ on $Y$ by means of 
\begin{eqnarray*}
& x \preceq y & \, \text{if and only if} \quad x(t) \le y(t) \text{ for each } t \in [0,1],\\
& x \prec y & \, \text{if and only if} \quad x(t) < y(t) \text{ for each } t \in [0,1].
\end{eqnarray*}
Then ${(Y,\|\, . \,\|_{\infty},\preceq,\prec)}$ is a solid Banach space.
This space is normal but nonregular. Consider, for example, the sequence ${(x_n)}$ in $Y$ defined by 
${x_n(t) = t^n}$.
We have ${x_1 \succeq x_2 \succeq \cdots \succeq 0}$ but ${(x_n)}$ is not convergent in $Y$.
\end{exmp}

\begin{exmp} \label{exmp:C^1[0,1]}
Let ${Y = C^1[0,1]}$ with the norm ${\|x\| = \|x\|_{\infty} + \|x'\|_{\infty}}$.
Define the ordering $\preceq$ and $\prec$ as in Example~\ref{exmp:C[0,1]}.
Then ${(Y,\|\, .\,\|, \preceq,\prec)}$ is is a solid Banach space.
The space $Y$ is not normal. Consider, for example, the sequences ${(x_n)}$ and ${(y_n)}$ in $Y$ defined by ${x_n(t) = \frac{t^n}{n}}$ and ${y_n(t) = \frac{1}{n}}$.
It is easy to see that ${0 \preceq x_n \preceq y_n}$ for all $n$, ${y_n \rightarrow 0}$ and ${x_n \not\rightarrow 0}$.
\end{exmp}

\begin{rem}
Let ${(Y,\preceq,\rightarrow)}$ be an arbitrary ordered vector space. Then the relation $\prec$ on $Y$ defined by
\begin{equation} \label{eq:StrictOrdering}
x \prec y \quad \text{if and only if} \quad x \preceq y \text{ and } x \neq y
\end{equation}
is a strict ordering on $Y$ and it always satisfies axioms {(S1)--(S4)} and 
properties {(S5)--(S10)}. 
However, it is not in general a strict vector ordering on $Y$. 
For example, from the uniqueness of strict vector ordering (Theorem~\ref{thm:StrictVectorOrdering-SolidCone}) it follows that $\prec$ defined by \eqref{eq:StrictOrdering} is not a strict vector ordering in the ordered vector spaces defined in Examples {\ref{exmp:R^n}--\ref{exmp:C^1[0,1]}}. 
\end{rem}

\section{Order topology on solid vector spaces}
\label{sec:OrderTopologyOnSolidVectorSpaces}

In this section, we show that every solid vector space can be endowed with 
an order topology $\tau$ and that 
${x_n \rightarrow x}$ implies ${x_n \stackrel{\tau}{\rightarrow} x}$. 
As a consequence we show that every convergent sequence in a solid vector space has a unique limit.

\begin{defn}
Let ${(Y,\preceq,\prec,\rightarrow)}$ be a solid vector space, and let ${a,b \in Y}$ be two vectors with ${a \prec b}$.
Then the set ${(a,b) = \{x \in Y : a \prec x \prec b\}}$
is called an \emph{open interval} in $Y$.
\end{defn}

It is easy to see that every open interval in $Y$ is an infinite set. Indeed, one can prove that
${a + \lambda \, (b - a) \in (a,b)}$ for all $\lambda \in \Rset$ with $0 < \lambda < 1$.

\begin{thm} \label{thm:OrderTopology1}
Let ${(Y,\preceq,\prec,\rightarrow)}$ be a solid vector space. Then the collection $\cal{B}$ of all open intervals in $Y$ is a basis for a Hausdorff topology $\tau$ on $Y$.
\end{thm}

\begin{proof}
One has to prove that $\cal{B}$ satisfies the requirements for a basis. First, note that every vector $x$ of $Y$ lies in at least one element of $\cal{B}$. Indeed, ${x \in (x-c, x+c)}$ for each vector ${c \succ 0}$.
Second, note that the intersection of any two open intervals contains another open interval, or is empty. Suppose  ${(a_1,b_1)}$ and ${(a_2,b_2)}$ are two elements of $\cal{B}$ and a vector $x$ lies in their intersection. 
Then ${b_i - x \succ 0}$ and ${x - a_i \succ 0}$ for ${i = 1, 2}$. It follows from (S12) that there exists a vector ${c \succ 0}$ such that 
${c \prec b_i - x}$ and ${c \prec x - a_i}$ for ${i = 1, 2}$. 
Hence, ${a_i \prec x-c}$ and ${x+c \prec b_i}$ for ${i = 1, 2}$. This implies that
${(x-c,x+c) \subset (a_1,b_1) \cap (a_2,b_2)}$.

It remains to show that the topology $\tau$ is Hausdorff. 
We shall prove that for all ${x,y \in X}$ with ${x \neq y}$ there exists ${c \succ 0}$ such that the intersection of the intervals ${(x-c,x+c)}$ and ${(y-c,y+c)}$ is empty. Assume the contrary. 
Then there exists 
${x,y \in X}$ with ${x \neq y}$ such that for every ${c \succ 0}$ the intersection of ${(x-c,x+c)}$ and ${(y-c,y+c)}$ is nonempty. Now let ${c \succ 0}$ be fixed. 
Hence, there is a vector $z \in Y$ satisfying ${x-c \prec z \prec x+c}$ and ${y-c \prec z \prec y+c}$. 
Therefore, ${-c \prec x - z \prec c}$ and ${-c \prec z - y \prec c}$.
Using (S10), we get
${-2c \prec x - y \prec 2c}$. Applying these inequalities to ${\frac{1}{2} c}$, we conclude that ${x - y \prec c}$ and ${y - x \prec c}$ for each ${c \succ 0}$.
Now it follows from (S11) that ${x \preceq y}$ and ${y \preceq x}$ which is a contradiction since  ${x \neq y}$. 
\end{proof}

Thanks to Theorem~\ref{thm:OrderTopology1} we can give the following definition.

\begin{defn} \label{df:OrderTopology}
Let ${(Y,\preceq,\prec,\rightarrow)}$ be a solid vector space. The topology $\tau$ on $Y$ with  basis formed by open intervals in $Y$ is called the \emph{order topology} on $Y$.
\end{defn}

\begin{rem} \label{rem:OrderTopology}
Let ${(Y,\preceq,\prec,\rightarrow)}$ be a solid vector space. 
It follows from the proof of Theorem~\ref{thm:OrderTopology1} that the collection
\[
{\cal{B'}} = \{(x-c,x+c) : x,c \in Y, c \succ 0\}
\]
is also a basis for the order topology $\tau$ on $Y$.
\end{rem}

\begin{thm} \label{thm:OrderTopology2}
Let ${(Y,\preceq,\prec,\rightarrow)}$ be a solid vector space and let $\tau$ be the order topology on $Y$. Then:
\begin{enumerate}
	\item For a sequence ${(x_n)}$ in $Y$, ${x_n \stackrel{\tau}{\rightarrow} x}$ 
	if and only if for every 	${c \succ 0}$ there exists ${N \in \Nset}$ such that
	${x - c \prec x_n \prec x - c}$ for all ${n > N}$.
	\item For a sequence ${(x_n)}$ in $Y$, ${x_n \rightarrow x}$ implies 
	${x_n \stackrel{\tau}{\rightarrow} x}$.
\end{enumerate}
\end{thm}

\begin{proof}
The first claim follows from Remark~\ref{rem:OrderTopology}. Let ${x_n \rightarrow x}$ and ${(a,b)}$ be a neighborhood of $x$. From ${a \prec x \prec b}$ and (S5), we conclude that
${x_n \in (a,b)}$ for for all but finitely many $n$. 
Hence, ${x_n \stackrel{\tau}{\rightarrow} x}$ which proves the second claim.  
\end{proof}

At the end of the next section we shall prove that the converse of the statement (ii) of Theorem~\ref{thm:OrderTopology2} holds true if and only if $Y$ is normal.

\begin{thm} \label{thm:ConvergenceSolidVectorSpace-UniqueLimit}
If ${(Y,\preceq,\prec,\rightarrow)}$ is a solid vector space, then the convergence on $Y$ has the following properties.
\begin{description}
	\item {\rm{(C6)}} Each convergent sequence in $Y$ has a unique limit.
	\item {\rm{(C7)}} Each convergent sequence in $Y$ is bounded.
\end{description}
\end{thm}

\begin{proof}
(C6) Let $(x_n)$ be a convergent sequence in $Y$.
Assume that there are ${x,y \in Y}$ such that ${x_n \rightarrow x}$ and 
${x_n \rightarrow y}$.
It follows from Theorem~\ref{thm:OrderTopology2} that 
${x_n \stackrel{\tau}{\rightarrow} x}$ and ${x_n \stackrel{\tau}{\rightarrow} y}$.
According to Theorem~\ref{thm:OrderTopology1} the topology $\tau$ is Hausdorff.  
Now by the uniqueness of the limit of a convergent sequence in the Hausdorff topological space 
${(Y,\tau)}$, we conclude that ${x=y}$.

(C7) Let $(x_n)$ be a convergent sequence in $Y$ and ${x_n \rightarrow x}$.
By Theorem~\ref{thm:OrderTopology2}, ${x_n \stackrel{\tau}{\rightarrow} x}$. 
Choose an open interval ${(a,b)}$ which contains $x$. 
Then there exists a natural number $N$ such that ${x_n \in (a,b)}$ for all $n \ge N$. 
According to (S13), the set ${\{a,b,x_1,\ldots,x_N\}}$ is bounded in $Y$
which proves that ${(x_n)}$ is bounded. 
\end{proof}

\section{Minkowski functional on solid vector spaces}
\label{sec:MinkowskiFunctionalOnSolidVectorSpaces}

In this section, using the Minkowski functional, we prove that the order topology on every solid vector space is normable. Also we show that every normal and solid vector space $Y$ is normable
in the sense that there exists a norm ${\|\, \, . \, \|}$ on $Y$ such that ${x_n \rightarrow x}$ if and only if 
${x_n \stackrel{\|.\,\|}{\rightarrow} x}$. 
Finally, we give a criterion for a normal vector space and show that the convergence of a sequence in normal and solid vector space has the properties of the convergence in $\Rset$.
This last result shows that the Sandwich theorem plays an important role in solid vector spaces. 

\begin{defn}
Let $Y$ be a real vector space. 
A subset $A$ of $Y$ is called:
\begin{description}
	\item (a) \emph{absorbing}, if for all ${x \in Y}$ there exists $\lambda > 0$ such that ${x \in \lambda \, A}$;
	\item (b) \emph{balanced}, if ${\lambda \, A \subset A}$ for every $\lambda \in \Rset$ with 	${|\,\lambda\,| \le 1}$;
	\end{description}
\end{defn}

\begin{defn} \label{df:MinkowskiFunctional}
Let $Y$ be a real vector space and ${A \subset Y}$ an absorbing set.
Then the functional ${\|\, \, . \, \| \, \colon Y \to \Rset}$ defined by
\begin{equation} \label{eq:MinkowskiFunctional-inf}
\|\,x\,\| = \inf \{ \lambda \ge 0 : x \in \lambda \, A \}.
\end{equation}
is called the \emph{Minkowski functional} of $A$.
\end{defn}

It is well known (see, e.g. \cite[Theorem 1.35]{Rud73}) that the Minkowski functional of every absorbing, convex and balanced subset $A$ of a vector space $Y$ is a seminorm on $Y$. 
    
\begin{lem} \label{lem:MinkowskiFunctional1}
Let ${(Y,\preceq,\rightarrow)}$ be an ordered vector space, and let ${A \subset Y}$ be an absorbing, convex, balanced and bounded set.
Then the Minkowski functional ${\|\, \, . \, \| \, \colon Y \to \Rset}$ of $A$ is a norm on $Y$. Moreover, if $A$ is closed, then
	\begin{equation} \label{eq:MinkowskiFunctional-min1}
	\|\,x\,\| = \min \{ \lambda \ge 0 : x \in \lambda \, A \}.
	\end{equation}
\end{lem}

\begin{proof}
Let $x \in Y$ be fixed and let ${B_x = \{ \lambda \ge 0 : x \in \lambda \, A \}}$. Since $A$ is absorbing, $B_x$ is nonempty. Since $A$ is balanced, ${\alpha \in B_x}$ and ${\alpha < \beta}$ imply ${\beta \in B_x}$. Let
${\inf {B_x} = \lambda}$.
By the definition of $\inf B_x$, for every $n \in \Nset$ there exists ${\alpha \in B}$ such that 
${\alpha < \lambda + \frac{1}{n}}$. Hence, 
${\lambda + \frac{1}{n} \in B_x}$ which means that 
${x \in \left( \lambda + \frac{1}{n} \right) A}$. 

Now we are ready to prove the ${\|\, \, . \, \|}$ is indeed a norm on $Y$. 
Since the Minkowski functional of $A$ is seminorm, we have only to prove that ${\|\,x\,\| = 0}$ implies ${x = 0}$. 
Let $x$ be a vector in $Y$ such that ${\|\,x\,\| = 0}$. In the case $\lambda =0$ the inclusion 
${x \in \left( \lambda + \frac{1}{n} \right) A}$ reduces to ${x \in \frac{1}{n} A}$. 
Since $A$ is bounded, there is an interval $[a,b]$ containing $A$. 
Hence, ${\frac{1}{n} \, a \preceq x \preceq \frac{1}{n} \, b}$ for all $n \in \Nset$. 
According to (C3), ${\frac{1}{n} \, a \rightarrow 0}$ and ${\frac{1}{n} \, b \rightarrow 0}$.
Hence, applying (V3) we conclude that ${0 \preceq x \preceq 0}$ which means that ${x=0}$. 

Now we shall prove \eqref{eq:MinkowskiFunctional-min1} provided that $A$ is closed.
We have to prove that ${\lambda = \inf B_x}$ belongs to $B_x$. 
If ${\lambda = 0}$, then ${x=0}$ which implies that ${\lambda \in B_x}$.
Now let ${\lambda \neq 0}$. 
The inclusion ${x \in \left( \lambda + \frac{1}{n} \right) A}$ implies that the sequence ${(x_n)}$ defined by 
${x_n = \left( \lambda + \frac{1}{n} \right)^{-1} x}$ lies in $A$. 
According to (C3), ${x_n \rightarrow \lambda^{-1} x}$ which implies
${x \lambda^{-1} \in A}$ since $A$ is closed. Hence, $x \in \lambda \, A$ which proves that 
${\lambda \in B_x}$.
\end{proof}

\begin{defn}
Let ${(Y,\preceq,\rightarrow)}$ be an ordered vector space, and let ${a,b \in Y}$ be two vectors with ${a \preceq b}$.
Then the set
${[a,b] = \{x \in Y : a \preceq x \preceq b\}}$
is called a \emph{closed interval} in $Y$.
\end{defn}

Obviously, every closed interval ${[-b,b]}$ in an ordered vector space $Y$ is a convex, balanced, closed and bounded set.
It follows from (S14) that ${[-b,b]}$ is also an absorbing set provided that $Y$ is a solid vector space and $b \succ 0$.

\begin{defn}
Let ${(Y,\preceq,\rightarrow)}$ be an ordered vector space.
A norm ${\|\, \, . \, \|}$ on $Y$ is called:
\begin{description}
	\item (a) \emph{monotone} if ${\|\,x\,\| \le \|\,y\,\|}$ 
	whenever ${0 \preceq x \preceq y}$.
	\item (b) \emph{semimonotone} if there exists a constant ${K > 0}$ such that 
	${\|\,x\,\| \le K \,\|\,y\,\|}$ 
	whenever ${0 \preceq x \preceq y}$. 
	\end{description}
\end{defn}

\begin{lem} \label{lem:MinkowskiFunctional-SolidCone}
Let ${(Y,\preceq,\prec,\rightarrow)}$ be a solid vector space. 
Let ${\|\, \, . \, \| \, \colon Y \to \Rset}$ be the Minkowski functional of ${[-b,b]}$ 
for some vector $b$ in $Y$ with ${b \succ 0}$. Then:
\begin{enumerate}
	\item ${\|\, \, . \, \|}$ is a monotone norm on $Y$ which can be defined by
	\begin{equation} \label{eq:MinkowskiFunctional-min2}
	\|\,x\,\| = \min \{ \lambda \ge 0 : -\lambda \, b \preceq x \preceq \lambda \, b \}.
	\end{equation}
 	\item  For ${x \in Y}$ and ${\varepsilon > 0}$,
	\begin{equation} \label{eq:MinkowskiFunctionalProperty1}
	\|\,x\,\| < \varepsilon 
	\quad\text{if and only if}\quad - \varepsilon \, b \prec x \prec \varepsilon \, b .
	\end{equation}
\end{enumerate}
\end{lem}

\begin{proof}
(i) The claim with the exception of the monotonicity of the norm follows from
Lemma~\ref{lem:MinkowskiFunctional1}.
Let $x$ and $y$ be two vectors in $Y$ such that ${0 \preceq x \preceq y}$. 
From \eqref{eq:MinkowskiFunctional-min2}, we get ${y \preceq \|\,y\,\| \, b}$. 
Hence, ${-\|\,y\,\| \, b \preceq x \preceq \|\,y\,\| \, b}$. Again from \eqref{eq:MinkowskiFunctional-min2}, we conclude that ${\|\,x\,\| \preceq \|\,y\,\|}$.
Hence, ${\|\, \, . \, \|}$ is a monotone norm.

(ii) Let ${\|\,x\,\| < \varepsilon}$. 
By \eqref{eq:MinkowskiFunctional-min2}, we have 
${- \|\,x\,\| \, b \preceq x \preceq \|\,x\,\| \, b}$ which implies
${- \varepsilon \, b \prec x \prec \varepsilon \, b}$.
 
Conversely, let ${- \varepsilon \, b \prec x \prec \varepsilon \, b}$. 
Then ${\varepsilon \, b - x \succ 0}$ and ${\varepsilon \, b + x \succ 0}$. 
It follows from (S12) that there is ${\lambda > 0}$ such that 
${\varepsilon \, b - x \succ \lambda \, b}$ and 
${\varepsilon \, b + x \succ \lambda \, b}$. Consequently, 
${- (\varepsilon - \lambda) \, b \prec x \prec (\varepsilon - \lambda) \, b}$. 
From this and \eqref{eq:MinkowskiFunctional-min2}, we conclude that
${\|\,x\,\| \le \varepsilon - \lambda < \varepsilon}$.
\end{proof}

The following theorem shows that the order topology on $Y$ is normable.
  
\begin{thm} \label{thm:MinkowskiFunctional-SolidCone}
Let ${(Y,\preceq,\prec,\rightarrow)}$ be a solid vector space, and let
${\|\, \, . \, \| \, \colon Y \to \Rset}$ be the Minkowski functional of ${[-b,b]}$ for some ${b \in Y}$ with  ${b \succ 0}$. Then:
\begin{enumerate}
	\item The monotone norm ${\|\, \, . \, \|}$ generates the order topology on $Y$.
	\item For a sequence ${(x_n)}$ in $Y$,
${x_n \rightarrow x}$ implies ${x_n \stackrel{\|.\,\|}{\rightarrow} x}$.
\end{enumerate}
\end{thm}

\begin{proof}
(i) Denoting by ${B(x,\varepsilon)}$ an open ball in the normed space 
${(Y,\|\, \, . \, \|)}$, we shall prove that each ${B(x,\varepsilon)}$ contains 
some interval ${(u,v)}$ in $Y$ and vice versa. 
First, we shall prove the following identity
\begin{equation} \label{eq:Ball=Interval}
	B(x,\varepsilon) = (x - \varepsilon \, b,x - \varepsilon \, b) 
	\quad\text{for all } x \in Y \text{ and } \varepsilon > 0 . 
\end{equation}
According to Lemma~\ref{lem:MinkowskiFunctional-SolidCone},
for each ${x,y \in Y}$and ${\varepsilon > 0}$, 
\[
\|x - y\| < \varepsilon 
\quad\text{if and only if}\quad x - \varepsilon \, b \prec y \prec x + \varepsilon \, b
\]
which proves \eqref{eq:Ball=Interval}. Note that identity \eqref{eq:Ball=Interval} means that every open ball in the normed space ${(Y,\|\, \, . \, \|)}$ is an open interval in
$Y$. Now let ${(u,v)}$ be an arbitrary open interval in $Y$ and let ${x \in (u,v)}$.
Choose an interval of the type ${(x-c,x+c)}$ which is a subset of ${(u,v)}$,
where ${c \in Y}$ with ${c \succ 0}$. Then choosing ${\varepsilon > 0}$ such that ${\varepsilon \, b \prec c}$, we conclude by \eqref{eq:Ball=Interval} that ${B(x,\varepsilon) \subset (u,v)}$.

(ii) follows from (i) and Theorem~\ref{thm:OrderTopology2}.
\end{proof}

The main part of Theorem~\ref{thm:MinkowskiFunctional-SolidCone} can be formulated in the following theorem. 

\begin{thm} \label{thm:NormOnSolidVectorCone}
Let ${(Y,\preceq,\prec,\rightarrow)}$ be a solid vector space. Then there exists a monotone norm ${\|\, \, . \, \|}$ on $Y$ such that the following statements hold true.
\begin{enumerate}
	\item The norm ${\|\, \, . \, \|}$ generates the order topology on $Y$.
	\item For a sequence ${(x_n)}$ in $Y$,
	${x_n \rightarrow x}$ implies ${x_n \stackrel{\|.\,\|}{\rightarrow} x}$.
\end{enumerate}
\end{thm}

In the next theorem we shall give a criterion for a normal vector space.
In particular, this theorem shows that every normal and solid vector space $Y$ is normable in the sense that there exists a norm ${\|\, \, . \, \|}$ on $Y$ such that 
${x_n \rightarrow x}$ if and only if ${x_n \stackrel{\|.\,\|}{\rightarrow} x}$. 
Analogous result for normability of normal topological vector space was proved by Vandergraft \cite{Van67}.

\begin{thm} \label{thm:NormalSolidVectorSpace}
Let ${(Y,\preceq,\prec,\rightarrow)}$ be a solid vector space.
Then the following statements are equivalent.
\begin{enumerate}
	\item $Y$ is a normal vector space
	\item The convergence in $Y$ is generated by a monotone norm on $Y$.
	\item The convergence in $Y$ is generated by the order topology on $Y$.
\end{enumerate}
\end{thm}

\begin{proof}
{$(i)\rightarrow(ii)$}. 
Suppose $Y$ be a normal vector space. 
Let ${\|\, \, . \, \| \, \colon Y \to \Rset}$ be the Minkowski functional of ${[-b,b]}$ 
for some vector $b$ in $Y$ with ${b \succ 0}$. 
According to Lemma~\ref{lem:MinkowskiFunctional-SolidCone} the Minkowski functional of ${[-b,b]}$ is a monotone norm on $Y$.  
We shall prove that the convergence in $Y$ is generated by this norm. We have to prove that for a sequence ${(x_n)}$ in $Y$, 
${x_n \rightarrow x}$ if and only if ${x_n \stackrel{\|.\,\|}{\rightarrow} x}$. 
Without loss of generality we may assume that ${x=0}$. 
Then we have to prove that ${x_n \rightarrow 0}$
if and only if ${\|\,x_n\| \rightarrow 0}$.
Taking into account Theorem~\ref{thm:MinkowskiFunctional-SolidCone} we have only to prove that  
${\|\,x_n\| \rightarrow 0}$ implies ${x_n \rightarrow 0}$.
Let ${\|\,x_n\| \rightarrow 0}$.
By Lemma~\ref{lem:MinkowskiFunctional-SolidCone}, we get
\[
- \|\,x_n\| \, b \preceq x_n \preceq \|\,x_n\| \, b
\quad\text{for all }\, n.
\]
It follows from axiom (C3) that ${\|\,x_n\| \, b \rightarrow 0}$. Then by the Sandwich theorem
we conclude that ${x_n \rightarrow 0}$.

{$(ii)\rightarrow(iii)$}.
Suppose the convergence in $Y$ is generated by a monotone norm ${\|\, \, . \, \|}$ on $Y$, i.e. for a sequence ${(x_n)}$ in $Y$, 
${x_n \rightarrow x}$ if and only if ${x_n \stackrel{\|.\,\|}{\rightarrow} x}$.
We shall prove that the convergence in $Y$ is generated by the order topology $\tau$ on $Y$.
According to Theorem~\ref{thm:OrderTopology2} it is sufficient to prove that for a sequence ${(x_n)}$ in $Y$, ${x_n \stackrel{\tau}{\rightarrow} x}$ implies ${x_n \rightarrow x}$.
Again without loss of generality we may assume that ${x=0}$. 
Let ${x_n \stackrel{\tau}{\rightarrow} 0}$. Let ${\varepsilon > 0}$ be fixed.
It follows from Theorem~\ref{thm:OrderTopology2} that for every vector ${c \succ 0}$,
\begin{equation} \label{eq:TauConvergence}
-c \prec x_n \prec c
\end{equation}
for all sufficiently large $n$. From \eqref{eq:TauConvergence}, we obtain
${0 \prec c - x_n \prec 2 \, c}$. By monotonicity of the norm, we conclude that
${\|\,c - x_n\| \le 2 \, \|\,c\,\|}$ which implies that 
${\|\,x_n\| \le 3 \, \|\,c\,\|}$. Now choosing a vector ${c \succ 0}$ such that  
${\|\,c\,\| < \varepsilon / 3}$, we obtain ${\|\,x_n\| < \varepsilon}$ for all sufficiently large $n$. Hence, ${\|\,x_n\| \rightarrow 0}$ which equivalent ${x_n \rightarrow x}$.

{$(iii)\rightarrow(i)$}.
Suppose the convergence in $Y$ is generated by the order topology on $Y$. 
We shall prove that $Y$ is normal.
Obviously, condition \eqref{eq:SandwichTheorem1} in Definition~\ref{df:NormalSpace} is equivalent to the following
\begin{equation} \label{eq:SandwichTheorem2}
0 \preceq x_n \preceq y_n \text{  for all  } n \, 
\text{ and }\, y_n \rightarrow 0  
\quad \text{imply} \quad x_n \rightarrow 0. 
\end{equation}
Let ${(x_n)}$ and ${(y_n)}$ be two sequences in $Y$ such that ${0 \preceq x_n \preceq y_n}$ for all $n$ and
${y_n \rightarrow 0}$. We have to prove that ${x_n \rightarrow 0}$.
Let ${c \succ 0}$ be fixed. It follows from ${y_n \rightarrow 0}$ and (S5) that ${y_n \prec c}$ for all but finitely many $n$.
From this and ${0 \preceq x_n \preceq y_n}$, we conclude that ${-c \prec x_n \prec c}$ for all sufficiently large $n$.
Now it follows from Theorem~\ref{thm:OrderTopology2} that 
${x_n \stackrel{\tau}{\rightarrow} x}$ which is equivalent to ${x_n \rightarrow 0}$.    
\end{proof}

Note that Theorem~\ref{thm:NormalSolidVectorSpace} remains true if we replace in it ``monotone norm'' by ``semimonotone norm''. 

The following theorem shows that the convergence in a normal and solid vector space
has the properties of the convergence in $\Rset$.  

\begin{thm} \label{thm:ConvergenceNormalSolidVectorSpace-Properties}
If ${(Y,\preceq,\prec,\rightarrow)}$ is a normal and solid vector space, then the convergence on $Y$ has the following additional properties. 
\begin{description}
	\item {\rm{(C8)}} Each subsequence of a convergent sequence converges to the same limit. 
	\item {\rm{(C9)}} The convergence of a sequence and its limit do not depend on 
	finitely many of its terms.
	\item {\rm{(C10)}} If ${\lambda_n \rightarrow \lambda}$ in $\Rset$ and 
	${x_n \rightarrow x}$, then ${\lambda_n\, x_n \rightarrow \lambda\, x}$.
	\item {\rm{(C11)}} If ${\lambda_n \rightarrow 0}$ in $\Rset$ and ${(x_n)}$ 
	is a bounded sequence in $Y$, then ${\lambda_n\, x_n \rightarrow 0}$.
	\item {\rm{(C12)}} If ${(\lambda_n)}$ is a bounded sequence in $\Rset$ and 
	${x_n \rightarrow 0}$, then ${\lambda_n\, x_n \rightarrow 0}$. 
	\item {\rm{(C13)}} For each sequence  ${(x_n)}$ in $Y$, ${x_n \rightarrow x}$ if and only if 
	for every ${c \succ 0}$ there exists a natural number $N$ such that
	${x - c \prec x_n \prec x - c}$ for all ${n > N}$.
\end{description}
\end{thm}

\begin{proof}
Let ${\|\, \, . \, \|}$ be a norm on $Y$ that generates the convergence in $Y$. 
The existence of such norm follows from Theorem~\ref{thm:NormalSolidVectorSpace}.
The properties C8)--(C10) are valid in any normed space. 
Property (C13) follows from Theorems \ref{thm:OrderTopology2} and \ref{thm:NormalSolidVectorSpace}.
The proofs of (C11) and (C12) are similar. We will prove only (C11).
Since ${(x_n)}$ is bounded, there exist $a,b \in Y$ such that ${a \preceq x_n \preceq b}$ for all $n$. This implies
\begin{equation} \label{eq:Bounded1}
|\lambda_n| \, a \preceq |\lambda_n| \, x_n \preceq |\lambda_n| \, b
\end{equation}
By axiom (C3), we get 
${|\lambda_n| \, a \rightarrow 0}$ and ${|\lambda_n| \, b \rightarrow 0}$.
Applying the Sandwich theorem to the inequalities \eqref{eq:Bounded1}, we conclude that
${|\lambda_n| \, x_n \rightarrow 0}$. Then by Theorem~\ref{thm:NormalSolidVectorSpace}, 
we obtain ${\|\,|\lambda_n| \,\, x_n \| \rightarrow 0}$, that is, 
${\|\,\lambda_n \, x_n \| \rightarrow 0}$. Again by Theorem~\ref{thm:NormalSolidVectorSpace}, we conclude that 
${\lambda_n \, x_n \rightarrow 0}$.    
\end{proof}

\section{Cone metric spaces and cone normed spaces}
\label{sec:ConeMetricSpacesConeNormedSpaces}

In this section we introduce the notions of cone metric spaces and cone normed spaces.
Cone metric spaces were first introduced in 1934 by Kurepa \cite{Kur34}.
Cone normed spaces were first introduced in 1936 by Kantorovich \cite{Kan36,Kan39}.
For more on these abstract metric spaces, see the monograph of Collatz \cite{Col64} 
and the survey  paper of Zabrejko \cite{Zab97}.

\begin{defn} \label{df:ConeMetricSpace}
Let $X$ be a nonempty set, and let ${(Y,\preceq,\rightarrow)}$ be an ordered vector space. 
A vector-valued function ${d \colon X \times X \to Y}$ is said to be a \emph{cone metric} 
on $Y$ if the following conditions hold:
\begin{enumerate}
	\item $d(x,y) \succeq 0$ for all $x,y \in X$ and $d(x,y) = 0$ if and only if $x=y$;
	\item $d(x,y) = d(x,y)$ for all $x,y \in X$;
	\item $d(x,y) \preceq d(x,z) + d(z,y)$ for all $x,y,z \in X$.
\end{enumerate}
The pair ${(X,d)}$ is called a \emph{cone metric space} over $Y$.
The elements of a cone metric space $X$ are called \emph{points}. 
\end{defn}

Obviously, every metric space is a cone metric space over $\Rset$.
In Section~\ref{sec:ConeMetricSpacesOverSolidVectorSpaces} 
we show that the theory of cone metric spaces over solid vector spaces
is very close to the theory of the metric spaces.   

\begin{defn} \label{df:ValuedField}
A map ${|\, \, . \, | \, \colon K \to \Rset}$ is called an \emph{absolute value} on a field $K$ if it satisfies the following axioms: 
\begin{enumerate}
\item $|\, x| \ge 0$ for all $x \in X$ and $|\, x| = 0$ if and only if $x = 0$;
\item $|\,x \, y| = |\,x| \, . |\,y|$ for all $x,y \in X$;
\item $|\,x + y| \le |\,x| + |\,y|$ for all $x \in X$.
\end{enumerate}
An absolute value is called \emph{trivial} if ${|\,x| = 1}$ for ${x \neq 0}$. A field ${(K,|\,\, . \,|)}$ equipped with a nontrivial absolute value is called a \emph{valued field}. 
\end{defn}

Note that finite fields and their extensions only have the trivial absolute value.
A valued field is always assumed to carry the topology induced by the metric 
${\rho(x, y) = |\,x - y|}$, with respect to which it is a topological field.
An absolute value is also called a \emph{multiplicative valuation} or a \emph{norm}.
For more on valuation theory, see Engler and Prestel \cite{EP05}.

One of the most important class of cone metric spaces is the class of cone normed spaces.

\begin{defn} \label{df:ConeNormedSpace}
Let $X$ be a vector space over a valued field ${(\Kset,|\,\, . \,|\,)}$,  
and let ${(Y,\preceq,\rightarrow)}$ be an ordered vector space. 
A map ${\|\, \, . \, \| \, \colon X \to Y}$ is said to be a \emph{cone norm} on $X$ 
if the following conditions hold: 
\begin{enumerate}
\item $\|\, x\| \succeq 0$ for all $x \in X$ and $\|\, x\| = 0$ if and only if $x = 0$;
\item $\|\,\lambda \, x\| = |\,\lambda| \, \|\,y\,\|$ 
for all ${\lambda \in \Kset}$ and $x \in X$;
\item $\|\,x + y\| \preceq \|\,x\,\| + \|\,y\,\|$ for all ${x, y \in X}$. 
\end{enumerate}
The pair ${(X,\|\, \, . \, \|)}$ is said to be a \emph{cone normed space} over $Y$.
\end{defn}

It is easy to show that every cone normed space ${(X,\|\, \, . \, \|)}$ over an ordered vector space $Y$ is a cone metric space over $Y$ with the cone metric defined by
${d(x, y)=\|\,x - y\|}$.

We end this section with the definitions of closed balls and bounded sets in cone metric spaces.

\begin{defn}
Let  ${(X,d)}$ be a cone metric space over an ordered vector space ${(Y,\preceq,\rightarrow)}$. For a point ${x_0 \in X}$ and a vector ${r \in Y}$ with 
${r \succeq 0}$, the set
\[
\overline{U}(x_0,r) = \{x \in X : d(x, x_0) \preceq r\}
\]
is called a \emph{closed ball} with center $x_0$ and radius $r$. 
\end{defn}

\begin{defn}
Let  $X$ be a cone metric space.
\begin{description}
	\item (a) A set ${A \subset X}$ is called \emph{bounded} if it is contained in some closed ball.
	\item (b) A sequence ${(x_n)}$ in $X$ is called \emph{bounded} if the set of its terms is bounded.
\end{description}
\end{defn}

Let ${(X,d)}$ be a cone metric space over an ordered vector space
${(Y,\preceq,\rightarrow)}$. 
It is easy to show that a nonempty set ${A \subset X}$ is bounded if and only if there exists a vector ${b \in Y}$ such that ${d(x,y) \preceq b}$ for all 
${x,y \in A}$.

Analogously, if ${(X,\|\, \, . \, \|)}$ is a cone normed space over an ordered vector space ${(Y,\preceq,\rightarrow)}$, then a nonempty set ${A \subset X}$ is bounded if and only if there exists a vector ${b \in Y}$ such that ${\|x\| \preceq b}$ for all 
${x \in A}$.

\section{Cone metric spaces over solid vector spaces}
\label{sec:ConeMetricSpacesOverSolidVectorSpaces}

In this section we shall study the cone metric spaces over solid vector spaces.
The theory of such cone metric spaces is very close to the theory of the usual metric spaces.
We show that every cone metric space over a solid vector space is a metrizable topological space. Every cone normed space over a solid vector space is normable.

\subsection{Topological structure of cone metric spaces}
\label{sec:TopologicalStructureOfConeMetricSpaces}

\begin{defn}
Let  ${(X,d)}$ be a cone metric space over a solid vector space ${(Y,\preceq,\prec,\rightarrow)}$. 
For a point ${x_0 \in X}$ and a vector ${r \in Y}$ with 
${r \succ 0}$, the set
\[
U(x_0,r) = \{x \in X : d(x,x_0) \prec r\} 
\]
is called an \emph{open ball} with center $x_0$ and radius $r$. 
\end{defn}

\begin{thm} \label{thm:TopologicalSpace}
Let  ${(X,d)}$ be a cone metric space over a solid vector space ${(Y,\preceq,\prec,\rightarrow)}$.
Then the collection 
\[
{\mathcal B} = \{ U(x,r) : x \in X, r \in Y, r \succ 0 \}
\]
of all open balls in $X$ is a basis for a topology $\tau_d$ on $X$
\end{thm}

\begin{proof}
Suppose  ${U(x_1,c_1)}$ and ${U(x_2,c_2)}$ are two open balls in $X$ 
and $x \in U(x_1,c_1) \cap U(x_2,c_2)$.
Then ${d(x, x_i) \prec c_i}$ for ${i = 1, 2}$.  From (S3), we get ${c_i - d(x, x_i) \succ 0}$ for ${i = 1, 2}$. It follows from (S12) that there exists a vector ${c \in Y}$ with ${c \succ 0}$ such that ${c \prec c_i - d(x, x_i)}$ for 
${i = 1, 2}$. By (S3), we obtain ${d(x, x_i) \prec c_i - c}$ for ${i = 1, 2}$. 
Now using the triangle inequality and (S10), it easy to show that
${U(x,c) \subset U(x_1,c_1) \cap U(x_2,c_2)}$.
Therefore, the collection $\mathcal{B}$ is a basis for a topology on $X$. 
\end{proof}

Thanks to Theorem~\ref{thm:TopologicalSpace} we can give the following definition.

\begin{defn} \label{df:ConeMetricTopology}
Let  ${(X,d)}$ be a cone metric space over a solid vector space ${(Y,\preceq,\prec,\rightarrow)}$.
The topology $\tau_d$ on $X$ with  basis formed by open balls in $X$ is called the 
\emph{cone metric topology} on $X$.
\end{defn}

We shall always assume that a cone metric space ${(X,d)}$ over a solid vector space $Y$
is endowed with the cone metric topology $\tau_d$. Hence, every cone metric space is a topological space.  

\begin{defn}\label{df:CauchySequence}
Let ${(X,d)}$ be a cone metric space over a solid vector space
${(Y,\preceq,\prec,\rightarrow)}$. Then:
\begin{enumerate}
	\item A sequence ${(x_n)}$ in $X$ is called \emph{Cauchy} if for every ${c \in Y}$ with ${c \succ 0}$ there is ${N \in \Nset}$ such that ${d(x_n,x_m) \prec c}$ for all 
${n, m > N}$.
	\item A cone metric space $X$ is called \emph{complete} if each Cauchy sequence in $X$ 
	is convergent.
	\item A complete cone normed space is called a \emph{cone Banach space}. 
	\end{enumerate}
\end{defn}

In the following theorem we show that each cone metric space ${(X,d)}$ over a  solid vector space is metrizable. Moreover, if ${(X,d)}$ is a complete cone metric space, then it is completely metrizable.    

\begin{thm} \label{thm:CMS-Metrizability}
Let ${(X,d)}$ be a cone metric space over a solid vector space 
${(Y,\preceq,\prec,\rightarrow)}$.
Suppose ${\|\, \, . \, \| \, \colon Y \to \Rset}$ is the Minkowski functional of ${[-b,b]}$ for some ${b \in Y}$ with  ${b \succ 0}$. 
Then:
\begin{enumerate}
	\item The metric ${\rho \colon X \times X \to \Rset}$ defined by 
	${\rho(x,y) = \|d(x,y)\|}$ generates the cone metric topology on $X$.
	\item The cone metric space ${(X,d)}$ is complete if and only if the metric space ${(X,\rho)}$ is complete.
	\item For ${x_i,y_i \in X}$ and $\lambda_i \in \Rset$ $(i=0,1,\ldots,n)$,
\[
d(x_0,y_0) \preceq \lambda_0 + \sum_{i=1}^n {\lambda_i \, d(x_i,y_i)}
\text{ implies  }
\rho(x_0,y_0) \le \|\lambda_0\| + \sum_{i=1}^n {\lambda_i \, \rho(x_i,y_i)}
\]
\end{enumerate}
\end{thm}

\begin{proof}
(i) It follows from Lemma~\ref{lem:MinkowskiFunctional-SolidCone}(i) and Definition~\ref{df:ConeMetricSpace} that $\rho$ is a metric on $X$.
Denoting by ${B(x,\varepsilon)}$ an open ball in the metric space ${(X,\rho)}$ and by   
${U(x, c)}$ an open ball in the cone metric space ${(X,d)}$, we shall prove that each ${B(x,\varepsilon)}$ contains some ${U(x, c)}$ and vice versa. 
First, we shall show that
\begin{equation} \label{eq:BallIdentity1}
	B(x,\varepsilon) = U(x,\varepsilon \, b) 
	\quad\text{for all } x \in X \text{ and } \varepsilon > 0 . 
\end{equation}
According to Lemma~\ref{lem:MinkowskiFunctional-SolidCone}(ii),
for all ${x,y \in X}$ and ${\varepsilon > 0}$, 
\[
\|d(x,y)\| < \varepsilon 
\quad\text{if and only if}\quad d(x,y) \prec \varepsilon \, b ,
\]
that is,
\begin{equation} \label{eq:BallIdentity}
\rho(x,y) < \varepsilon \quad\text{if and only if}\quad d(x,y) \prec \varepsilon \, b
\end{equation}
which proves \eqref{eq:BallIdentity1}. Note that identity \eqref{eq:BallIdentity1} means that every open ball in the metric space ${(X,\rho)}$ is an open ball in the cone metric space
${(X,d)}$. Now let ${U(x, c)}$ be an arbitrary open ball in the cone metric space ${(X,d)}$.
Choosing ${\varepsilon > 0}$ such that ${\varepsilon \, b \prec c}$, we conclude by \eqref{eq:BallIdentity1} that ${B(x,\varepsilon) \subset U(x,c)}$.

(ii) Let ${(x_n)}$ be a sequence in $X$. We have to prove that ${(x_n)}$ is $d$-Cauchy if and only if it is $\rho$-Cauchy. First note that \eqref{eq:BallIdentity} implies that for each ${\varepsilon > 0}$ and all ${m,n \in \Nset}$,
\[
\rho(x_n,x_m) < \varepsilon 
\quad\text{if and only if}\quad d(x_n,x_m) \prec \varepsilon \, b .
\]
Let ${(x_n)}$ be $d$-Cauchy and ${\varepsilon > 0}$ be fixed. 
Then there is an integer $N$ such that ${d(x_n,x_m) \prec \varepsilon b}$ for  all 
${m,n > N}$. Hence, ${\rho(x_n,x_m) < \varepsilon}$ for  all ${m,n > N}$ which means that ${(x_n)}$ be $\rho$-Cauchy . 

Now, let ${(x_n)}$ be $\rho$-Cauchy and ${c \succ 0}$ be fixed. 
Choose ${\varepsilon > 0}$ such that ${\varepsilon b \prec c}$.
Then there is an integer $N$ such that ${d(x_n,x_m) < \varepsilon}$ for  all 
${m,n > N}$. Therefore, for these $n$ and $m$ we get 
${d(x_n,x_m) \prec \varepsilon b \prec c}$ which means that ${(x_n)}$ is $d$-Cauchy.

(iii) follows from the monotony of the norm ${\|\, \, . \, \|}$ and the definition of the metric $\rho$.
\end{proof}

As we have seen the identity \eqref{eq:BallIdentity1} plays an important role
in the proof of Theorem~\ref{thm:CMS-Metrizability}.
It is easy to see that this identity holds also for closed balls
in the spaces ${(X,\rho)}$ and ${(X,d)}$. Namely, we have
\begin{equation} \label{eq:BallIdentity2}
	\overline{B}(x,\varepsilon) = \overline{U}(x,\varepsilon \, b) 
	\quad\text{for all } x \in X \text{ and } \varepsilon > 0 . 
\end{equation}

The main idea of Theorem~\ref{thm:CMS-Metrizability} can be formulated in the following theorem.

\begin{thm} \label{thm:CMS-Metrizability2}
Let ${(X,d)}$ be a cone metric space over a solid vector space 
${(Y,\preceq,\prec,\rightarrow)}$. Then there exists a metric $\rho$ on $X$ such that the following statements hold true.
\begin{enumerate}
	\item The metric $\rho$ generates the cone metric topology on $X$.
	\item The cone metric space ${(X,d)}$ is complete if and only if the metric space ${(X,\rho)}$ is complete.
	\item For ${x_i,y_i \in X}$ $(i=0,1,\ldots,n)$ 
	and ${\lambda_i \in \Rset}$ $(i=1,\ldots,n)$,
\[
d(x_0,y_0) \preceq \sum_{i=1}^n {\lambda_i \, d(x_i,y_i)}
\,\text{ implies  }\,
\rho(x_0,y_0) \le \sum_{i=1}^n {\lambda_i \, \rho(x_i,y_i)}
\]
\end{enumerate}
\end{thm}

Metrizable topological spaces inherit all topological properties from metric spaces. 
In particular, it follows from Theorem~\ref{thm:CMS-Metrizability2} that every cone metric space over a solid vector space
is a Hausdorff paracompact space and first-countable.
Since every  first countable space is sequential, we immediately get that every cone metric space is a sequential space. Hence, as a consequence of Theorem~\ref{thm:CMS-Metrizability2}
we get the following corollary. 

\begin{cor} \label{cor:Sequential}
Let ${(X,d)}$ be a cone metric space over a solid vector space $Y$. 
Then the following statements hold true. 
\begin{enumerate}
	\item A subset of $X$ is open if and only if it is sequentially open. 
	\item A subset of $X$ is closed if and only if it is sequentially closed.
	\item A function ${f \colon D \subset X \to X}$ is continuous if and only if it is sequentially continuous.
\end{enumerate}
\end{cor}

\begin{lem} \label{lem:ClosedBalls}
Let ${(X,d)}$ be a cone metric space over a solid vector space
${(Y,\preceq,\prec,\rightarrow)}$. Then every closed ball ${\overline{U}(a,r)}$ in $X$ is a closed set. 
\end{lem}

\begin{proof}
According to Corollary~\ref{cor:Sequential} we have to prove that ${\overline{U}(a,r)}$ is a sequentially closed set.
Let ${(x_n)}$ be a convergent sequence in ${\overline{U}(a,r)}$ and let $x \in Y$ be its limit.
Let ${c \in Y}$ with ${c \succ 0}$ be fixed. 
Since ${x_n \rightarrow x}$, then there exists $n \in \Nset$ such that ${d(x_n,x) \prec c}$. 
Using the triangle inequality, we get $d(x,a) \preceq d(x_n,a) + d(x_n,x) \prec r + c$. 
Hence, $d(x,a) - r \prec c$ for all ${c \in Y}$ with ${c \succ 0}$.
Then by (S11) we conclude that $d(x,a) - r \preceq 0$ which implies ${x \in \overline{U}(a,r)}$. 
Therefore, ${\overline{U}(a,r)}$ is a closed set in $X$.
\end{proof}

\begin{rem}
Theorem~\ref{thm:CMS-Metrizability2} plays an important role in the theory of
cone metric spaces over a solid vector space. 
In particular, using this theorem one can prove that some fixed point theorems in cone metric spaces are equivalent to their versions in usual  metric spaces. 
For example, the short version of the Banach contraction principle in complete cone metric spaces 
(see Theorem~\ref{thm:CMS-BanachContractionPrinciple-ShortVersion} below) follows 
directly from its short version in metric spaces. 
Du \cite{Du10} was the first who showed that there are equivalence between some metric and cone metric results.
He obtained his results using the so-called nonlinear scalarization function.
One year later, Kadelburg, Radenovi\'c and Rako\v cevi\'c \cite{KRR11} showed that the same results can be obtained using Minkowski functional in topological vector spaces.  
\end{rem}

\begin{rem}
Theorem~\ref{thm:CMS-Metrizability} generalizes and extends some recent results of
Du \cite[Theorems 2.1 and 2.2]{Du10},
Kadelburg, Radenovi\'c and Rako\v cevi\'c \cite[Theorems 3.1 and 3.2]{KRR11},
\c Cakalli, S\"onmez and Gen\c c \cite[Theorem~2.3]{CSG10},
Simi\'c \cite[Theorem~2.2]{Sim11},
Abdeljawad and Rezapour \cite[Theorem~16]{AR11}
Arandelovi\'c and Ke\v cki\'c \cite[Lemma~2]{AK11},
All of these authors have studied cone metric spaces over a solid Hausdorff topological vector space.
Note that the identity \eqref{eq:BallIdentity1} was proved by \c Cakalli, S\"onmez and Gen\c c \cite[Theorem~2.2]{CSG10} provided that $Y$ is a Hausdorff topological vector space. 

Theorem~\ref{thm:CMS-Metrizability} generalizes and extends also some recent results of
Amini-Harandi and Fakhar \cite[Lemma~2.1]{AF10},
Turkoglu and Abuloha \cite{TA10},
Khani and Pourmahdian \cite[Theorem~3.4]{KP11},
S\"onmez \cite[Theorem~1]{Son10},
Asadi, Vaezpour and Soleimani \cite[Theorem~2.1]{AVS11},
Feng and Mao \cite[Theorem~2.2]{FM10}.
These authors have studied cone metric spaces over a solid Banach space.

Note that Asadi and Soleimani \cite{AS11} proved that the metrics of Feng and Mao \cite{FM10} and Du \cite{Du10} are equivalent.

Finally, let us note a work of Khamsi \cite{Kha10} in which he introduced a metric type structure in cone metric spaces over a normal Banach space.
\end{rem}

\begin{defn} \label{df:ConeNormTopology}
Let ${(X,\|\, \, . \, \|)}$ be a cone normed space over a solid vector space ${(Y,\preceq,\prec,\rightarrow)}$. 
The cone metric topology $\tau_d$ on $X$ induced by the metric ${d(x, y)=\|\,x - y\|}$
is called the \emph{cone topology} on $X$.
\end{defn}

In the following theorem we show that each cone normed space 
${(X,\|\, \, . \, \|)}$ over a solid vector space is normable. 
Moreover, if ${(X,\|\, \, . \, \|)}$ is a cone Banach space, then it is completely normable. 

\begin{thm} \label{thm:CNS-Normability}
Suppose $X$ is a vector space over a valued field ${(\Kset,|\,\, . \,|\,)}$.  
Let ${(X,\|\, \, . \, \|)}$ be a cone normed space over a solid vector space ${(Y,\preceq,\prec,\rightarrow)}$.
Let ${\mu \colon Y \to \Rset}$ be the Minkowski functional of ${[-b,b]}$ for some ${b \in Y}$ with  ${b \succ 0}$. Then:
\begin{enumerate}
	\item The norm ${|||\,\, . \,|||}$ defined by 
	${|||\,x||| = \mu(\|\,x\,\|)}$ generates the cone topology on $X$. 
	\item The space ${(X,\|\, \, . \, \|)}$ is a cone Banach space if and only if 
	${(X,|||\,\, . \,|||)}$ is a Banach space.
	\item For ${x_i \in X}$ and $\lambda_i \in \Rset$ $(i=0,1,\ldots,n)$,
\[
\|\,x_0\| \preceq \lambda_0 + \sum_{i=1}^n {\lambda_i \, \|\,x_i\|}
\,\,\,\text{implies}\,\,\,
|||\,x_0||| \le \mu(\lambda_0) + \sum_{i=1}^n {\lambda_i \, |||\,x_i|||}.
\]
\end{enumerate}
\end{thm}

\begin{proof}
The cone topology on the space ${(X,\|\, \, . \, \|)}$ is induced by the cone metric 
${d(x,y) = \|x - y\|}$ and the topology on ${(X,|||\,\, . \,|||)}$ is induced by the metric 
${\rho(x,y) = |||x - y|||}$. It is easy to see that ${\rho = \mu \circ d}$.
Now the conclusions of the theorem follow from Theorem~\ref{thm:CMS-Metrizability}.  
\end{proof}

\begin{rem}
Theorem~\ref{thm:CNS-Normability}(i) was recently proved by
\c Cakalli, S\"onmez and Gen\c c \cite[Theorem 2.4]{CSG10} 
provided that ${\Kset = \Rset}$ and $Y$ is a Hausdorff topological vector space.
\end{rem}

The following corollary is an immediate consequence of Theorem~\ref{thm:CNS-Normability}(i).

\begin{cor} \label{cor:ConeNormedSpaces-TVS}
Every cone normed space ${(X,\|\, \, . \, \|)}$ over a solid vector space $Y$ is a topological vector space. 
\end{cor}

\subsection{Convergence in cone metric spaces}
\label{sec:ConvergenceInConeMetricSpaces}

Let ${(X,d)}$ be a cone metric space over a solid vector space ${(Y,\preceq,\prec,\rightarrow)}$.
Let ${(x_n)}$ be a sequence in $X$ and $x$ a point in $X$. 
We denote the convergence of ${(x_n)}$ to $x$ with respect to the cone metric topology, by ${x_n \stackrel{d}{\rightarrow} x}$ or simply by ${x_n \rightarrow x}$.
Obviously, ${x_n \stackrel{d}{\rightarrow} x}$ if and only if for every vector ${c \in Y}$ with ${c \succ 0}$, ${d(x_n,x) \prec c}$ for all but finitely many $n$.
This definition for the convergence in cone metric spaces over a solid Banach space can be found in the works of Chung \cite{Chu81,Chu82} published in the period from 1981 to 1982.
The definition of complete cone metric space (Definition~\ref{df:CauchySequence}) in the case when $Y$ is a solid Banach space also can be found in \cite{Chu81,Chu82}.    

\begin{thm} \label{thm:ConvergenceInConeMetricSpaces}
Let ${(X,d)}$ be a cone metric space over a solid vector space 
$(Y,\preceq,\prec,\rightarrow)$. 
Then the convergence in $X$ has the following properties.
\begin{enumerate}
	\item Any convergent sequence has a unique limit.
	\item Any subsequence of a convergent sequence converges to the same limit.
	\item Any convergent sequence is bounded.
	\item The convergence and the limit of a sequence do not depend on finitely many of its terms.
\end{enumerate}
\end{thm}

\begin{proof}
The properties (i), (ii) and (iv) are valid in any Hausdorff topological space. 
It remains to prove (iii). 
Let ${(x_n)}$ be a sequence in $X$ which converges  to a point ${x \in X}$. 
Choose a vector ${c_1 \in Y}$ with ${c_1 \succ 0}$. 
Then there exists ${N \in \Nset}$ such that ${d(x_n,x) \prec c_1}$ for all ${n \ge N}$.
By (S13), there is a vector ${c_2 \in Y}$ such that ${d(x_n,x) \prec c_2}$ for all 
${n = 1,\ldots,N}$. Again by (S13), we get that there is a vector ${c \in Y}$ such that
${c_i \prec c}$ for ${i = 1,2}$. Then by the transitivity of $\prec$, we conclude that
${x_n \in U(x,c)}$ for all ${n \in \Nset}$ which means that ${(x_n)}$ is bounded.    
\end{proof}

Applying Theorem~\ref{thm:CMS-Metrizability}, we shall prove a useful sufficient condition for convergence of a sequence in a cone metric space over a solid vector space. 

\begin{thm} \label{thm:CMS-PropertyConvergence}
Let ${(X,d)}$ be a cone metric space over a solid vector space 
${(Y,\preceq,\prec,\rightarrow)}$.
Suppose ${(x_n)}$ is a sequence in $X$ satisfying
\begin{equation} \label{eq:CMS-PropertyConvergence1}
d(x_n,x) \preceq b_n + \alpha \, d(y_n,y) + \beta \, d(z_n,z)
\quad \text{for all }\, n,
\end{equation}
where $x$ is a point in $X$, 
${(b_n)}$ is a sequence in $Y$ converging to $0$, 
${(y_n)}$ is a sequence in $X$ converging to $y$, 
${(z_n)}$ is a sequence in $X$ converging to $z$,
$\alpha$ and $\beta$ are nonnegative real numbers. 
Then the sequence $(x_n)$ converges to $x$. 
\end{thm}

\begin{proof}
Let ${\|\, \, . \, \|}$ be the Minkowski functional of ${[-b,b]}$ for some ${b \in Y}$ with ${b \succ 0}$. Define the metric ${\rho}$ on $X$ as in Theorem~\ref{thm:CMS-Metrizability}. 
Then from \eqref{eq:CMS-PropertyConvergence1}, we get
\begin{equation} \label{eq:CMS-PropertyConvergence2}
\rho(x_n,x) \preceq \|\,b_n\| + \alpha \, \rho(y_n,y) + \beta \, \rho(z_n,z)
\quad \text{\emph{for all} }\, n,
\end{equation}
According to Theorem~\ref{thm:MinkowskiFunctional-SolidCone}(ii),
${b_n \rightarrow 0}$ implies ${\|\,b_n\| \rightarrow 0}$.
Hence, the right-hand side of \eqref{eq:CMS-PropertyConvergence2} converges to $0$ in $\Rset$.
By usual Sandwich theorem, we conclude that ${x_n \stackrel{\rho}{\rightarrow} x}$ which is equivalent to ${x_n \stackrel{d}{\rightarrow} x}$.
\end{proof}

\begin{rem}
A special case ${(\alpha = \beta = 0)}$ of Theorem~\ref{thm:CMS-PropertyConvergence} was 
given without proof by Kadelburg, Radenovi\'c and Rako\v cevi\'c \cite{KRR09} in the case when $Y$ is a Banach space.
This special case was proved by {\c S}ahin and Telsi \cite[Lemma~3.3]{ST10}.
\end{rem}

It is easy to see that if ${(x_n)}$ is a sequence in a cone metric space ${(X,d)}$ 
over a solid vector space $Y$, then
\begin{equation} \label{eq:CMS-PropertyConvergence3}
d(x_n,x) \rightarrow 0 \quad \text{implies} 
\quad x_n \stackrel{d}{\rightarrow} x,
\end{equation}
but the converse is not true (see Example~\ref{NoncontinuousConeMetric}(ii) below).
Note also that in general case the cone metric is not (sequentially) continuous function (see Example~\ref{NoncontinuousConeMetric}(iii) below), that is, from ${x_n \rightarrow x}$ and ${y_n \rightarrow y}$ it need not follow that
${d(x_n,y_n) \rightarrow d(x,y)}$.

In the following theorem we shall prove that the converse of \eqref{eq:CMS-PropertyConvergence3} holds provided that $Y$ is normal and solid.

\begin{thm} \label{thm:CMS/NormalSolid}
Let ${(X,d)}$ be a cone metric space over a normal and solid vector space  ${(Y,\preceq,\prec,\rightarrow)}$.
Then
\begin{equation} \label{eq:CMS/NormalSolid}
x_n \stackrel{d}{\rightarrow} x \quad\text{if and only if}\quad 
d(x_n,x) \rightarrow 0.
\end{equation}
\end{thm}

\begin{proof}
Let ${\|\, \, . \, \|}$ be the Minkowski functional of ${[-b,b]}$ for some ${b \in Y}$ with  ${b \succ 0}$. 
Define the metric ${\rho}$ on $X$ as in Theorem~\ref{thm:CMS-Metrizability}.
By Theorem~\ref{thm:CMS-Metrizability},
\begin{equation} \label{eq:CMS/NormalSolid1}
x_n \stackrel{d}{\rightarrow} x \quad\text{if and only if}\quad 
x_n \stackrel{\rho}{\rightarrow} x.
\end{equation}
By  Theorem~\ref{thm:MinkowskiFunctional-SolidCone}, for each sequence ${u_n}$ in $Y$
\[ 
u_n \rightarrow 0 \quad\text{if and only if}\quad \|\,u_n\| \rightarrow 0 .
\]
Applying this with ${u_n = d(x_n,x)}$, we get
\[ 
d(x_n,x) \rightarrow 0 \quad\text{if and only}\quad \rho(x_n,x) \rightarrow 0 ,
\]
that is,
\begin{equation} \label{eq:CMS/NormalSolid2}
d(x_n,x) \rightarrow 0 \quad\text{if and only}\quad x_n \stackrel{\rho}{\rightarrow} x.
\end{equation}
Now \eqref{eq:CMS/NormalSolid} follows from \eqref{eq:CMS/NormalSolid1} and
\eqref{eq:CMS/NormalSolid2}.
\end{proof}

The following theorem follows immediately from Corollary~\ref{cor:ConeNormedSpaces-TVS}.
It can also be proved by Theorem~\ref{thm:CMS-PropertyConvergence}. 
 
\begin{thm} \label{thm:ConvergenceInConeNormedSpaces}
Suppose $X$ is a vector space over a valued field ${(\Kset,|\,\, . \,|\,)}$.  
Let ${(X,\|\, \, . \, \|)}$ be a cone normed space over a solid vector space ${(Y,\preceq,\prec,\rightarrow)}$. Then the convergence in $X$ satisfies the properties 
\emph{(}i\emph{)}--\emph{(}iv\emph{)} of Theorem~\ref{thm:ConvergenceInConeMetricSpaces} and it satisfies also the following properties.
\begin{description}
	\item {\rm{(v)}} If $x_n \rightarrow x$ and $y_n \rightarrow y$, then 
	$x_n + y_n \rightarrow x+y$.
	\item {\rm{(vi)}} If $\lambda_n \rightarrow \lambda$ in $\Kset$ and $x_n \rightarrow x$, then $\lambda_n x_n \rightarrow \lambda\, x$.
\end{description}
\end{thm}

\subsection{Complete cone metric spaces}
\label{sec:CompleteConeMetricSpaces}

Now we shall prove a useful sufficient condition for Cauchy sequence in cone metric spaces
over a solid vector space. The second part of this result gives an error estimate for the limit of a convergent sequence in cone metric space. Also we shall prove a criterion for completeness of a cone metric space over a solid vector space.

\begin{thm} \label{thm:CMS-PropertiesCauchy}
Let ${(X,d)}$ be a cone metric space over a solid vector space
${(Y,\preceq,\prec,\rightarrow)}$. 
Suppose ${(x_n)}$ is a sequence in $X$ satisfying
\begin{equation} \label{eq:CMS-PropertiesCauchy}
d(x_n,x_m) \preceq b_n \quad\text{for all \,} n, m \ge 0 \text{ with } m \ge n,
\end{equation}
where ${(b_n)}$ is a sequence in $Y$ which converges to $0$. Then:
\begin{enumerate}
	\item The sequence ${(x_n)}$ is a Cauchy sequence in $X$.
	\item If ${(x_n)}$ converges to a point ${x \in X}$, 
	then 
\begin{equation} \label{eq:K-PropertiesConvergence}
d(x_n,x) \preceq b_n \quad \text{for all }\, n \ge 0. 
\end{equation}
\end{enumerate}
\end{thm}

\begin{proof}
(i) Let ${c \in Y}$ with ${c \succ 0}$ be fixed. According to (S5), 
${b_n \rightarrow 0}$ implies that there is ${N \in \Nset}$ such that 
${b_n \prec c}$ for all ${n >N}$. It follows from \eqref{eq:CMS-PropertiesCauchy} and (S2) that ${d(x_n,x_m) \prec c}$ for all ${m, n >N}$ with ${m \ge n}$. Therefore, $x_n$ is a Cauchy sequence in $X$.

(ii) Suppose ${x_n \rightarrow x}$. Let ${n \ge 0}$ be fixed. 
Choose an arbitrary ${c \in Y}$ with ${c \succ 0}$. 
Since ${x_n \rightarrow x}$, then there exists ${m >n}$ such that ${d(x_m,x) \prec c}$.
By the triangle inequality, \eqref{eq:CMS-PropertiesCauchy} and (S10), we get
\[
d(x_n,x) \preceq d(x_n,x_m) + d(x_m,x) \prec b_n + c.  
\]
It follows from (S3) that ${d(x_n,x) - b_n \prec c}$ holds for each ${c \succ 0}$. 
Which according to (S11) means that ${d(x_n,x) - b_n \preceq 0}$. 
Hence, ${d(x_n,x) \preceq  b_n}$ which completes the proof.      
\end{proof}

\begin{rem}
Theorem~\ref{thm:CMS-PropertiesCauchy}(i) was 
proved by  Azam, Beg and Arshad \cite[Lemma~1.3]{ABA10} in the case when $Y$ is a topological vector space. Note also that whenever the cone metric space ${(X,d)}$ is complete, then the assumption of the second part of Theorem~\ref{thm:CMS-PropertiesCauchy} is satisfied automatically.   
\end{rem}

A sequence of closed balls ${(\overline{U}(x_n,r_n))}$
in a cone metric space $X$ is called a \emph{nested sequence} 
if
\[ 
\overline{U}(x_1,r_1) \supset \overline{U}(x_2,r_2) \supset \ldots
\]
Now we shall prove a simple criterion for the completeness of a cone metric space over a solid vector space.

\begin{thm} [Nested ball theorem] \label{thm:NestedBallTheorem}
A cone metric space ${(X,d)}$ over a solid vector space ${(Y,\preceq,\prec,\rightarrow)}$
is complete if and only if every nested sequence ${(\overline{U}(x_n,r_n))}$ of closed balls in $X$ such that $r_n \rightarrow 0$ has a nonempty intersection. 
\end{thm}

\begin{proof}
Let ${\|\, \, . \, \|}$ be the Minkowski functional of ${[-b,b]}$ for some ${b \in Y}$ with ${b \succ 0}$. 
Define the metric ${\rho}$ on $X$ as in Theorem~\ref{thm:CMS-Metrizability}. By Theorem~\ref{thm:CMS-Metrizability}, ${(X,d)}$ is complete if and only if ${(X,\rho)}$ is complete.

\emph{Necessity}.
If ${(\overline{U}(x_n,r_n))}$ is a nested sequence of closed balls in ${(X,d)}$ 
such that ${r_n \rightarrow 0}$, then according to Lemma~\ref{lem:ClosedBalls} it is a nested sequence of closed sets in 
${(X,\rho)}$ with the sequence of diameters ${(\delta_n)}$ converging to zero.
Indeed, it easy to see that ${\rho(x,y) = \|d(x,y)\| \preceq 2 \, \|\,r_n\|}$ for all 
${x,y \in \overline{U}(x_n,r_n)}$.
Hence, ${\delta_n \le 2 \, \|\,r_n\|}$ which yields ${\delta_n \rightarrow 0}$. 
Applying Cantor's intersection theorem to the metric space ${(X,\rho)}$, 
we conclude that the intersection of the sets ${\overline{U}(x_n,r_n)}$ is nonempty.

\emph{Sufficiently}.
Assume that every nested sequence of closed balls in ${(X,d)}$ with radii converging to zero has a nonempty intersection. We shall prove that each nested sequence ${(\overline{B}(x_n,\varepsilon_n))}$ of closed balls in ${(X,\rho)}$ 
such that ${\varepsilon_n \rightarrow 0}$ has a nonempty intersection.
By identity \eqref{eq:BallIdentity2}, we get
\begin{equation}
	\overline{B}(x_n,\varepsilon_n) = \overline{U}(x_n,r_n) 
\quad\text{for all } n ,
\end{equation}
where ${r_n = \varepsilon_n \, b \rightarrow 0}$.
Hence, according to the assumptions the balls ${\overline{B}(x_n,\varepsilon_n)}$ have a nonempty intersection. Applying the nested ball theorem to the metric space ${(X,\rho)}$, we conclude that it is complete and so ${(X,d)}$ is also complete.
\end{proof}

\subsection{Examples of complete cone metric spaces}
\label{sec:ExamplesOfCompleteMetricSpaces}

We end this section with three examples of complete cone metric spaces. 
Some other examples on cone metric spaces can be found in \cite{Zab97}. 

\begin{exmp} \label{exmp:DiscreteConeMetricSpace}
Let $X$ be a nonempty set and let ${(Y,\preceq,\prec,\rightarrow)}$ be a  solid vector space. Suppose $a$ is a vector in $Y$ such that ${a \succeq 0}$ and ${a \neq 0}$. 
Define the cone metric ${d \colon X \times X \to Y}$ by
\begin{equation} \label{eq:DiscreteConeMetric}
d(x,y) = \left \{ \begin{array}{ll}
{a} & \mbox{ if $\, x \neq y$},\\
{0} & \mbox{ if $\, x = y$}.
\end{array} \right.
\end{equation}
Then ${(X,d)}$ is a complete cone metric space over $Y$. 
This space is called a \emph{discrete cone metric space}.
\end{exmp}

\begin{proof}
It is obvious that ${(X,d)}$ is a cone metric space 
(even if $Y$ is an arbitrary ordered vector space).
We shall prove that every Cauchy sequence in $X$ is stationary. 
Assume the contrary and choose a sequence ${(x_n)}$ in $X$ which is Cauchy but not stationary. 
Then for every ${c \in Y}$ with ${c \succ 0}$ there exist ${n, m \in \Nset}$ such that 
${d(x_n,x_m) \prec c}$ and ${x_n \neq x_m}$. Hence, ${a \prec c}$ for each ${c \succ 0}$.  
Then by (S11) we conclude that ${a \preceq 0}$ which together with ${a \succeq 0}$ leads to the contradiction 
${a = 0}$. Therefore, every Cauchy sequence in $X$ is stationary and so convergent in $X$.
\end{proof}

\begin{exmp} \label{NoncontinuousConeMetric}
Let ${(Y,\preceq,\prec,\rightarrow)}$ be a solid vector space, and let $X$ be its positive cone.
Define the cone metric ${d \colon X \times X \to Y}$ as follows
\begin{equation}
d(x,y) = \left \{ \begin{array}{ll}
{x+y} & \mbox{ if $\, x \neq y$},\\
{0} & \mbox{ if $\, x = y$}.
\end{array} \right.
\end{equation}
Then the following statements hold true:
\begin{enumerate}
	\item ${(X,d)}$ is a complete cone metric space over $Y$.
	\item If $Y$ is not normal, then there are sequences ${(x_n)}$ in $X$ such that
	${x_n \rightarrow 0}$ but ${d(x_n,0) \not\rightarrow 0}$.
	\item If $Y$ is not normal, then the cone metric $d$ is not continuous.
\end{enumerate}
\end{exmp}

\begin{proof}
First we shall prove the following claim: A sequence ${(x_n)}$ in $X$ is Cauchy if and only if it satisfies one of the following two conditions.
\begin{description}
	\item (a) The sequence ${(x_n)}$ is stationary.
	\item (b) For every ${c \succ 0}$ the inequality ${x_n \prec c}$ holds for all but finitely many $n$.
\end{description}
\emph{Necessity}. Suppose ${(x_n)}$ is Cauchy but not stationary.
Then for every ${c \in Y}$ with ${c \succ 0}$ there exists ${N \in \Nset}$ such that ${d(x_n,x_m) \prec c}$ for all 
${n, m > N}$. Hence, for all 
${n, m > N}$ we have ${x_n + x_m \prec c}$ whenever ${x_n \neq x_m}$.
Let ${n > N}$ be fixed. Since ${(x_n)}$ is not stationary, there exists ${m > N}$
such that ${x_n \neq x_m}$. Hence, ${x_n + x_m \prec c}$. From this taking into account that ${x_m \succeq 0}$, 
we get ${x_n \prec c}$ and so ${(x_n)}$ satisfies (b). 

\emph{Sufficiently}. Suppose that ${(x_n)}$ satisfies (b). 
Then for every ${c \succ 0}$ there exists ${N \in \Nset}$ such that for all 
${n > N}$ we have ${x_n \prec \frac{1}{2}\, c}$.
Let ${n, m > N}$ be fixed.
Then ${d(x_n,x_m) \preceq x_n + x_m \prec c}$ which means that ${(x_n)}$ is Cauchy.  

Now we shall prove the statements of the example.

(i)  Let ${(x_n)}$ be a Cauchy sequence in $X$. 
If ${(x_n)}$ satisfies (a), then it is convergent.
Now suppose that ${(x_n)}$ satisfies (b). 
Let ${c \succ 0}$ be fixed. Then
${d(x_n,0) \preceq x_n \prec c}$ for all but finitely many $n$. This proves that
${x_n \rightarrow 0}$. Therefore, in both cases ${(x_n)}$ is convergent. 

(ii) Since $Y$ is not normal, then there exist two sequences ${(x_n)}$ and ${(y_n)}$ in $Y$ such that 
${0 \preceq x_n \preceq y_n}$ for all $n$, 
${y_n \rightarrow 0}$ and ${x_n \not\rightarrow 0}$.
Let us consider ${(x_n)}$ as a sequence in $X$.
It follows from the definition of the cone metric $d$ that ${d(x_n,0) = x_n}$ for all $n$.
Hence, ${d(x_n,0) \preceq y_n}$ for all $n$. 
Then by Theorem~\ref{thm:CMS-PropertyConvergence}, 
we conclude that ${x_n \rightarrow 0}$.
On the other hand ${d(x_n,0) = x_n \not\rightarrow 0}$.

(iii) Assume that the cone metric $d$ is a continuous.
Let ${(x_n)}$ be any sequence in $X$ satisfying (ii). By ${x_n \rightarrow 0}$ and continuity of $d$, we obtain ${d(x_n,0) \rightarrow d(0,0)}$, i.e. ${x_n \rightarrow 0}$ in $Y$ 
which is a contradiction. Hence, the cone metric $d$ is not continuous.
\end{proof}

\begin{exmp} \label{N-DimensionalConeBanachSpace}
Let ${X = \Kset^n}$ be $n$-dimensional vector space over $\Kset$, where ${(\Kset,|\,\, . \,|\,)}$ is a valued field. 
Let ${Y = \Rset^n}$ be $n$-dimensional real vector space with the coordinate-wise convergence and the coordinate-wise ordering (see Example~\ref{exmp:R^n}). 
Define the cone norm ${\|\, \, . \, \| \colon X  \to Y}$ by
\begin{equation}
	\|\,x\,\| = (\alpha_1 |x_1|,\,\ldots,\,\alpha_n |x_n|),
\end{equation}
where ${x = (x_1,\,\ldots,\,x_n)}$ and ${\alpha_1,\,\ldots,\,\alpha_n}$ are positive real numbers.
Then ${(X,\|\, \, . \, \|)}$ is a cone Banach space over $Y$. 
\end{exmp}

\section{Iterated contractions in cone metric spaces}
\label{sec:IteratedContractionsInConeMetricSpaces}

The iterated contraction principle in usual metric spaces was first mentioned in 1968 
by Rheinboldt \cite{Rhe68} as a special case of a more general theorem. 
Two years later, an explicit formulation of this principle (with a posteriori error estimates) was given in the monograph of 
Ortega and Rheinboldt \cite[Theorem {12.3.2}]{OR70}. 
Great contribution to the iterated contraction principle in metric spaces and its applications to the fixed point theory was also given by Hicks and Rhoades \cite{HR79}, Park \cite{Par80} and others 
(see Proinov \cite[Section~6]{Pro10}). 

In this section we shall establish a full statement of the iterated contraction principle
in cone metric spaces. We shall formulate the result for nonself mappings since the case of selfmappings is a special case of this one.

Let ${(X,d)}$ be a cone metric space over a solid vector space ${(Y,\preceq,\prec,\rightarrow)}$, and let ${T \colon D \subset X \to X}$ be an arbitrary mapping in $X$.
Then starting from a point ${x_0 \in D}$ we can build up the Picard iterative sequence
\begin{equation}\label{eq:PicardIteration0}
x_{n + 1}  = Tx_n ,  \quad n = 0,1,2,\ldots,
\end{equation}
associated to the mapping $T$. We say that the iteration \eqref{eq:PicardIteration0} is \emph{well defined} if 
${x_n \in D}$  for all  ${n = 0,1,2,\ldots}$ 
The main problems which arise for the Picard iteration are the following: 
\begin{enumerate}
	\item \textsc{Convergence problem}. To find initial conditions for ${x_0 \in D}$ which guarantee that the Picard iteration \eqref{eq:PicardIteration0} is well defined and converging to a point ${\xi \in D}$.
	\item \textsc{Existence problem}. To find conditions which guarantee that 
	$\xi$ is a fixed points of $T$.
	\item \textsc{Uniqueness problem}. To find a subset of $D$ in which $\xi$ is a unique fixed point of $T$. 
	\item \textsc{Error estimates problem}. To find a priory and a posteriori estimates for the cone distance ${d(x_n,\xi)}$. 
\end{enumerate}

In our opinion, the solving of problem (i) for the convergence of the Picard iteration plays an important role for the solving of problem (ii) for existence of fixed points of $T$. 
It turns out that in many cases the convergence of the Picard iteration to a point 
${\xi \in D}$ implies that $\xi$ is a fixed point of $T$. For example, such situation can be seen in the next proposition. 

\begin{prop} \label{prop:CMS-ExistenceFP}
Let ${(X,d)}$ be a cone metric space over a solid vector space ${(Y,\preceq,\prec,\rightarrow)}$, and let 
${T \colon D \subset X \to X}$. Suppose that for some ${x_0 \in D}$ Picard iteration \eqref{eq:PicardIteration0} is well defined and converging to a point ${\xi \in D}$. Then each of the following conditions implies that $\xi$ is a fixed point of $T$. 
\begin{description}
		\item {\rm{(F1)}} $T$ is continuous at $\xi$.
		\item {\rm{(F2)}} $T$ has a closed graph.
		\item {\rm{(F3)}} $G(x) = \|d(x,Tx)\|$ is lower semicontinuous at $\xi$ for some 
		semimonotone norm ${\|\, \, . \, \|}$ on $Y$.
		\item {\rm{(F4)}} $d(\xi,T\xi) \preceq \alpha \, d(x,\xi) + \, \beta \, d(Tx,\xi)$ 
		for each $x \in D$, where $\alpha, \beta \ge 0$.
\end{description}
\end{prop}

\begin{proof}
If (F1) or (F2) is satisfied, then the conclusion follows from Theorem~\ref{thm:ConvergenceInConeMetricSpaces} and definition \eqref{eq:PicardIteration0} of the Picard iteration. 

Suppose that condition (F3) holds. 
Since the norm ${\|\, \, . \, \|}$ is semimonotone, there exists a constant 
${K > 0}$ such that ${\|\,x\,\| \le K \,\|\,y\,\|}$ whenever ${0 \preceq x \preceq y}$.
First we shall prove that 
${x_n \rightarrow \xi}$ implies ${\|d(x_n,x_{n+1})\| \rightarrow 0}$.
We claim that for every ${\varepsilon > 0}$ there exists a vector ${c \in Y}$ such that 
${c \succ 0}$ and ${\|\,c\,\| < \varepsilon}$. To prove this take a vector ${b \in Y}$ with ${b \succ 0}$. We have
${\left\|\,\frac{1}{n}\, b\,\right\| = \frac{1}{n} \left\|\,b\,\right\| \rightarrow 0}$.  
Hence, every vector ${c = \frac{1}{n}\, b}$ with sufficiently large $n$ satisfies ${\|\,c\,\| < \varepsilon}$. 
Now let ${\varepsilon > 0}$ be fixed. Choose a vector 
${c \in Y}$ such that ${c \succ 0}$ and ${\|\,c\,\| < \varepsilon / K}$.  
From the triangle inequality, we get ${d(x_n,x_{n+1}) \preceq d(x_n,\xi) + d(x_{n+1},\xi)}$. 
Now it follows from ${x_n \rightarrow \xi}$ that ${d(x_n,x_{n+1}) \prec c}$ for all but finitely many $n$. 
Hence, ${\|d(x_n,x_{n+1})\| \preceq K \|\,c\,\| < \varepsilon}$ for these $n$. 
Therefore, ${\|d(x_n,x_{n+1})\| \rightarrow 0}$.
Now taking into account that $G$ is lower semicontinuous at $\xi$ 
we conclude that
\[
0 \le \|\,d(\xi, T\xi)\| = G(\xi) \le \liminf_{n \rightarrow \infty}{G(x_n)} =
 \liminf_{n \rightarrow \infty}{\|\,d(x_n, x_{n+1}\|} =0 
\]
which implies that $\xi$ is a fixed point of $T$.

Suppose that condition (F4) is satisfied. By substituting ${x = x_n}$, we get
\[ 
d(\xi,T\xi) \preceq \alpha \, d(x_n,\xi) + \beta \, d(x_{n+1},\xi).
\]
From this, taking into account that ${x_n \rightarrow \xi}$, 
we conclude that ${d(\xi,T\xi) \prec c}$ for each ${c \in Y}$ with ${c \succ 0}$. 
According to (S11), this implies ${d(\xi,T\xi) \preceq 0}$. 
Therefore, ${d(\xi,T\xi) = 0}$ which means that $\xi$ is a fixed point of $T$.     
\end{proof}

\begin{rem}
Obviously, if the space ${(X,d)}$ in Proposition~\ref{prop:CMS-ExistenceFP} is a metric space, then the function $G$ in (F4) can be defined by ${G(x) = d(x,Tx)}$. 
In a metric space setting this is a classical result (see \cite{HR79}).  
Let us note also that if the space $Y$ in Proposition~\ref{prop:CMS-ExistenceFP} is a normal and solid normed space with norm ${\|\, \, . \, \|}$, then one can choose in (F4) just this norm (see \cite{War09}).
\end{rem}

Throughout this and next section for convenience we assume in $\Rset$ that ${0^0 =1}$ by definition.

\begin{prop} \label{prop:CMS-IteratedContractionPrinciple}
Let ${(X,d)}$ be a cone metric space over a solid vector space ${(Y,\preceq,\prec,\rightarrow)}$. 
Suppose ${(x_n)}$ is a sequence in $X$ satisfying  
\begin{equation} \label{eq:IteratedSequence}
	d(x_{n+1},x_{n+2}) \preceq \lambda\, d(x_n,x_{n+1}) 
	\quad \text{for every } n \ge 0,
\end{equation}
where $0{ \le \lambda < 1}$. Then ${(x_n)}$ is a Cauchy sequence in $X$ and 
lies in the closed ball ${\overline{U}(x_0,r)}$ with radius
\[
r = \frac{1}{1-\lambda} \, d(x_0,x_1).
\]
Moreover, if ${(x_n)}$ converges to a point $\xi$ in $X$, then the following estimates hold:
\begin{equation} \label{eq:IteratedSequence-AprioriEstimate}
	d(x_n, \xi) \preceq \frac{\lambda^n}{1-\lambda} \, d(x_0, x_1) 
	\quad\mbox{for all } n \ge 0;
\end{equation}
	\begin{equation} \label{eq:IteratedSequence-AposterioriEstimate1}
	d(x_n, \xi) \preceq \frac{1}{1-\lambda} \, d(x_n, x_{n+1}) 
	\quad\mbox{for all } n \ge 0;
\end{equation}
\begin{equation} \label{eq:IteratedSequence-AposterioriEstimate2}
	d(x_n, \xi) \preceq \frac{\lambda}{1-\lambda} \, d(x_n, x_{n-1}) 
	\quad\mbox{for all } n \ge 1.
\end{equation}
\end{prop}

\begin{proof}
From \eqref{eq:IteratedSequence} by induction on ${n \ge 0}$, we get
\[
	d(x_n,x_{n+1}) \preceq \lambda^n \, d(x_0,x_1) 
	\quad \text{for every } n \ge 0.
\]
Now we shall show that ${(x_n)}$ satisfies
\begin{equation} \label{eq:Iterated Sequence-CauchyProperty}
d(x_n,x_m) \preceq b_n \quad\text{for all \,} n, m \ge 0 \text{ with } m \ge n,
\end{equation}
where $b_n = \frac{\lambda^n}{1-\lambda} d(x_0,x_1)$. Indeed, for all 
${n, m \ge 0}$ with ${m \ge n}$, we have 
\begin{eqnarray*}
d(x_n,x_m)
& \preceq & \sum_{j=n}^m {d(x_j,x_{j+1})} 
\preceq \sum_{j=n}^m {\lambda^j \, d(x_0,x_1)}
= \left( \sum_{j=n}^m {\lambda^j} \right) d(x_0,x_1)\\
& \preceq & \left( \sum_{j=n}^{\infty} {\lambda^j} \right) d(x_0,x_1)
= \frac{\lambda^n}{1-\lambda} \, d(x_0,x_1) = b_n \,. 
\end{eqnarray*}
It follows from axiom (C3) that ${b_n \rightarrow 0}$ in $Y$. Then by Theorem~\ref{thm:CMS-PropertiesCauchy}(i) we conclude that ${(x_n)}$ is a Cauchy sequence in $X$.
Putting ${n=0}$ in \eqref{eq:Iterated Sequence-CauchyProperty} we obtain that 
${d(x_m,x_0) \preceq b_0}$ for every ${m \ge 0}$. Hence, the sequence ${(x_n)}$ lies in the ball ${\overline{U}(x_0,r)}$ since ${r = b_0}$.
Now suppose that ${(x_n)}$ converges to a point ${\xi \in X}$. Then it follows from Theorem~\ref{thm:CMS-PropertiesCauchy}(ii) that ${(x_n)}$ satisfies the inequality ${d(x_n,\xi) \preceq b_n}$ (for every ${n \ge 0}$) which proves \eqref{eq:IteratedSequence-AprioriEstimate}.
Applying \eqref{eq:IteratedSequence-AprioriEstimate} with ${n=0}$, we conclude that the first two terms of the sequence ${(x_n)}$ satisfy the inequality
\[
d(x_0,\xi) \preceq \frac{1}{1-\lambda} \, d(x_0,x_1).
\]
Note that for every ${n \ge 0}$ the sequence ${(x_n,x_{n+1},x_{n+2},\ldots)}$ also satisfies \eqref{eq:IteratedSequence} and converges to $\xi$. Therefore, applying the last inequality to the first two terms of this sequence we get \eqref{eq:IteratedSequence-AposterioriEstimate1}.    
The inequality \eqref{eq:IteratedSequence-AposterioriEstimate2} follows from
\eqref{eq:IteratedSequence-AposterioriEstimate1} and \eqref{eq:IteratedSequence}.
\end{proof}

\begin{rem}
Proposition~\ref{prop:CMS-IteratedContractionPrinciple} generalizes, improves and complements a recent result of 
Latif and Shaddad \cite[Lemma~3.1]{LS10}. They have proved that a sequence ${(x_n)}$ in a cone metric space $X$ satisfying \eqref{eq:IteratedSequence} is Cauchy provided that $Y$ is a normal Banach space. 
\end{rem}

\begin{thm} [Iterated contraction principle] \label{thm:CMS-IterationContractionPrinciple2}
Let ${(X,d)}$ be a complete cone metric space over a solid vector space ${(Y,\preceq,\prec,\rightarrow)}$. 
Suppose ${T \colon D \subset X \to X}$ be a mapping satisfying the following conditions:
\begin{description}
	\item {\rm{(a)}} $d(Tx, T^2 x) \preceq \lambda \, d(x, Tx)$ for all $x \in D$ with $Tx \in D$, where ${0 \le \lambda < 1}$.
	\item {\rm{(b)}} There is ${x_0 \in D}$ such that $\overline{U}(x_0,r) \subset D$, where
	$r = \frac{1}{1-\lambda} \, d(x_0,Tx_0)$.
\end{description}
Then the following hold true:  
\begin{enumerate}
		\item \textsc{Convergence of the iterative method}. The Picard iteration
		\eqref{eq:PicardIteration0} starting from $x_0$ is well defined, remains 
		in the closed ball $\overline{U}(x_0,r)$ and converges to a point
		$\xi \in \overline{U}(x_0,r)$.
	\item \textsc{A priori error estimate}. The following estimate holds:
\begin{equation} 
	d(x_n, \xi) \preceq \frac{\lambda^n}{1 - \lambda} \, d(x_0, Tx_0) 
	\quad\mbox{for all } n \ge 0.
\end{equation}
	\item \textsc{A posteriori error estimates}. The following estimates hold:
	\begin{equation}
	d(x_n, \xi) \preceq \frac{1}{1-\lambda} \, d(x_n, x_{n+1}) 
	\quad\mbox{for all } n \ge 0;
\end{equation}
\begin{equation}
	d(x_n, \xi) \preceq \frac{\lambda}{1 - \lambda} \, d(x_n, x_{n-1}) 
	\quad\mbox{for all } n \ge 1.
\end{equation}
	\item \textsc{Existence of fixed points}. If at least one of the conditions
	\emph{(}F1\emph{)}--\emph{(}F2\emph{)} is satisfied, then $\xi$ is a fixed point of $T$.
\end{enumerate}
\end{thm}

\begin{proof}
Define the function ${\rho \colon D \rightarrow X}$ by 
${\rho(x) = \frac{1}{1-\lambda} \, d(x, Tx)}$.
It follows from condition (a) that ${\rho(Tx) \preceq \lambda \, \rho(x)}$ for each 
${x \in D}$. Now define the set $U$ as follows
\[
U = \{ x \in D : \overline{U}(x,\rho(x)) \subset D \}. 
\]
It follows from $\rho(x_0) = r$ and (b) that the set $U$ is not empty. 
We shall prove that ${T(U) \subset U}$. 
Let $x$ be a given point in $U$. 
It follows from the definition of $\rho$ that ${d(x, Tx) \preceq \rho(x)}$ which means that 
${Tx \in \overline{U}(x,\rho(x)) \subset D}$. Therefore, ${Tx \in D}$.
Further, we shall show that
\[
\overline{U}(Tx,\rho(Tx)) \subset \overline{U}(x,\rho(x)).
\]
Indeed, suppose that $y \in \overline{U}(Tx,\rho(Tx))$. Then
\[
d(y,x) \preceq d(y,Tx) + d(x,Tx) \preceq \rho(Tx) + d(x,Tx)
\preceq \lambda \rho(x) + d(x,Tx) = \rho(x) 
\]
which means that ${y \in \overline{U}(x,\rho(x))}$.
Hence, ${\overline{U}(Tx,\rho(Tx)) \subset D}$ and so ${Tx \in U}$. This proves that
${T(U) \subset U}$ which means that Picard iteration $(x_n)$ is well defined.
From (a), we deduce that it satisfies \eqref{eq:IteratedSequence}. 
Now conclusions (i)--(iii) follow from Proposition~\ref{prop:CMS-IteratedContractionPrinciple}.
Conclusion (iv) follows from Proposition~\ref{prop:CMS-ExistenceFP}.
\end{proof}

\begin{rem} \label{rem:Sefmapings}
Obviously, whenever $T$ is a selfmapping of $X$, 
condition (b) of Theorem~\ref{thm:CMS-IterationContractionPrinciple2} is satisfied automatically for every ${x_0 \in X}$ and so it can be omitted.
If $D$ is closed and ${T(D) \subset D}$, then condition (b) also can be dropped.
\end{rem}

\begin{rem}
Note that Theorem~\ref{thm:CMS-IterationContractionPrinciple2}(i) generalizes and extends some results of Pathak and Shahzad \cite[Theorem 3.7]{PS09} and Wardowski \cite[Theorem 3.3]{War09}. Their  results have been proved for a selfmapping 
$T$ of $X$ in the case when $Y$ is a normal Banach space.
\end{rem}

\section{Contraction mappings in cone metric spaces}
\label{sec:ContractionMappingsInConeMetricSpaces}

In 1922 famous Polish mathematician Stefan Banach \cite{Ban22} established his famous fixed point theorem nowadays known as the \emph{Banach fixed point theorem} or the \emph{Banach contraction principle}.
The Banach contraction principle is one of the most useful theorem in the fixed point theory. It has a short and complete statement. Its complete form in metric space setting can be seen for example in monographs of 
Kirk \cite{Kir90}, Zeidler \cite[Section 1.6]{Zei95} and Berinde \cite[Section 2.1]{Ber07}. 

In a cone metric space setting full statement of the Banach contraction principle for a nonself mapping is given by the following theorem.

\begin{thm} [Banach contraction principle] \label{thm:CMS-BanachContractionPrinciple}
Let ${(X,d)}$ be a complete cone metric space over a solid vector space ${(Y,\preceq,\prec,\rightarrow)}$.
Let ${T \colon D \subset X \to X}$ be a mapping satisfying the following conditions:
\begin{description}
	\item {\rm{(a)}} $d(Tx, Ty) \preceq \lambda \, d(x, y)$ for all $x,y \in D$, 
	where ${0 \le \lambda < 1}$.
	\item {\rm{(b)}} There is ${x_0 \in D}$ such that $\overline{U}(x_0,r) \subset D$, where
	$r = \frac{1}{1-\lambda} \, d(x_0, Tx_0)$.
\end{description}
Then the following hold true:  
\begin{enumerate}
		\item \textsc{Existence and uniqueness}. $T$ has a unique fixed point $\xi$ in $D$.
		\item \textsc{Convergence of the iterative method}. The Picard iteration
		\eqref{eq:PicardIteration0} starting from $x_0$ is well defined, remains 
		in the closed ball $\overline{U}(x_0,r)$ and converges to $\xi$.
	\item \textsc{A priori error estimate}. The following estimate holds:
\begin{equation} \label{eq:K-ApprioriEstimate2}
	d(x_n, \xi) \preceq \frac{\lambda^n}{1 - \lambda} \, d(x_0, Tx_0) 
	\quad\mbox{for all } n \ge 0.
\end{equation}
	\item \textsc{A posteriori error estimates}. The following estimates hold:
	\begin{equation}
	d(x_n, \xi) \preceq \frac{1}{1-\lambda} \, d(x_n, x_{n+1}) 
	\quad\mbox{for all } n \ge 0;
\end{equation}
\begin{equation}
	d(x_n, \xi) \preceq \frac{\lambda}{1 - \lambda} \, d(x_n, x_{n-1}) 
	\quad\mbox{for all } n \ge 1.
\end{equation}
	\item \textsc{Rate of convergence}. The rate of convergence of the 
	Picard iteration is given by
\begin{equation}
	d(x_{n+1}, \xi) \preceq \lambda \, d(x_n, \xi)
	\quad\mbox{for all } n \ge 1;
\end{equation}
\begin{equation}
	d(x_{n}, \xi) \preceq \lambda^n \, d(x_0, \xi) 
	\quad\mbox{for all } n \ge 0.
\end{equation}
\end{enumerate}
\end{thm}

\begin{proof}
Using the triangle inequality and the contraction condition (a), one can see that condition (F4) holds with 
${\alpha = \lambda}$ and ${\beta = 1}$. 
Conclusions (i)--(iv) with the exception of the uniqueness of the fixed point follow immediately from Theorem~\ref{thm:CMS-IterationContractionPrinciple2} since every contraction mapping is an iterated contraction mapping.
Suppose $T$ has two fixed points $x,y \in D$. Then it follows from (a) that
${d(x, y) \preceq \lambda \, d(x, y)}$ which leads to ${(1 - \lambda)d(x, y) \preceq 0}$
and so ${d(x, y) \preceq 0}$. Hence, ${d(x, y)= 0}$ which means that ${x=y}$. Therefore,
$\xi$ is a unique fixed point of $T$ in $D$. Conclusion (v) follows from (a) and (i) by putting ${x = x_n}$ and ${y = \xi}$.   
\end{proof}

Kirk in his paper \cite{Kir90} wrote for the Banach contraction principle in usual metric spaces the following ``The great significance of Banach's principle, and the reason it is one of the most frequently cited fixed point theorems in all of analysis, lies in the fact that
(i)--(v) contain elements of fundamental importance to the theoretical and practical treatment of mathematical equations''.
We would add that in general cone metrics give finer estimates than usual metrics.

Recall that a selfmapping $T$ of a cone metric space ${(X,d)}$ is called \emph{contraction} on $X$ if there exists 
${0 \le \lambda < 1}$ such that ${d(Tx, Ty) \preceq \lambda \, d(x, y)}$ for all ${x,y \in X}$.
The following short version of the Banach contraction principle for selfmappings in cone metric spaces follows immediately from Theorem~\ref{thm:CMS-BanachContractionPrinciple}. Note that the short version of Banach's principle follows also from the short version of Banach's principle in metric spaces and Theorem~\ref{thm:CMS-Metrizability}.  

\begin{thm} \label{thm:CMS-BanachContractionPrinciple-ShortVersion}
Each contraction $T$ on a cone metric space $(X,d)$ over a solid vector space $Y$ has a unique fixed point and for each $x_0 \in X$ the Picard iteration \eqref{eq:PicardIteration0} converges 
to the fixed point.
\end{thm}

\begin{rem}
Theorem~\ref{thm:CMS-BanachContractionPrinciple-ShortVersion} was proved by 
Huang and Zhang \cite[Theorem~1]{HZ07} in the case when $Y$ is a normal Banach space. 
One year later, Rezapour and Hamlbarani \cite[Theorem~2.3]{RH08} improved their result omitting the assumption of normality.
Finally, Du \cite[Theorem~2.3]{Du10} proved this result assuming that $Y$ 
is a locally convex Hausdorff topological vector space.

Recently, Radenovi\'c and Kadelburg \cite[Theorem~3.3]{RK11} have established the a priory estimate \eqref{eq:K-ApprioriEstimate2} for a selfmappings $T$ of a cone metric space $X$ over a solid Banach space $Y$.
\end{rem}

\section{Conclusion}
\label{sec:Conclusion}

In the first part of this paper (Sections 2--7) we develop a unified theory for solid vector spaces.
A real vector space $Y$ with convergence ($\rightarrow$) is called a \emph{solid vector space} if 
it is equipped with a vector ordering ($\preceq$) and a strict vector ordering ($\prec$). 
It turns out that every convergent sequence in a solid vector space has a unique limit.
Every solid vector space $Y$ can be endowed with an order topology $\tau$ such that 
${x_n \rightarrow x}$ implies ${x_n \stackrel{\tau}{\rightarrow} x}$.
It turns out that the converse of this implication holds if and only if the space $Y$ is normal, i.e. the Sandwich theorem holds in $Y$.
Using the Minkowski functional, we show that the order topology on every solid vector space is normable with a monotone norm. 
Among the other results in this part of the paper, we show that an ordered vector space can be equipped with a strict vector ordering if and only if it has a solid positive cone. Moreover, if the positive cone of the vector ordering is solid, then there exists a unique strict vector ordering on this space.

In the second part of the paper (Sections 8--9) we develop a unified theory for cone metric spaces and cone normed spaces over a solid vector space.
We show that every (complete) cone metric space $(X,d)$ over a solid vector space $Y$ is a (completely) metrizable topological space. Moreover, there exists an equivalent metric $\rho$ on $X$ that preserve some inequalities. In particular, an inequality of the type
\begin{equation} \label{eq:SpecialTypeInequality1}
d(x_0,y_0) \preceq \sum_{i=1}^n {\lambda_i \, d(x_i,y_i)} \qquad (x_i \in X, \lambda_i \in \Rset)
\end{equation}
implies the inequality
\begin{equation} \label{eq:SpecialTypeInequality2}
\rho(x_0,y_0) \preceq \sum_{i=1}^n {\lambda_i \, \rho(x_i,y_i)}.
\end{equation}
Using this result one can prove that some fixed point theorems in cone metric spaces are equivalent to their versions in usual  metric spaces. For example, the short version of the Banach contraction principle in a cone metric space is equivalent to its version in a metric space because the Banach contractive condition 
$d(Tx,Ty) \preceq \lambda \, d(x,y)$ is of the type \eqref{eq:SpecialTypeInequality1}.
Let us note that the above mentioned result cannot be applied to many contractive conditions in a cone metric space.
That is why we need further properties of cone metric spaces.
Further, we give some useful properties of cone metric spaces which allow us to prove convergence results for Picard iteration with a priori and a posteriori error estimates. 
Among the other results in this part of the paper we prove that every cone normed space over a solid vector space is normable.

In the third part of the paper (Sections 8--9) applying the cone metric theory  we present full statements of the iterated contraction principle and the Banach contraction principle in cone metric spaces over a solid vector space. 

Let us note that some of the results of the paper (Theorems \ref{thm:CMS-Metrizability},
\ref{thm:CNS-Normability}, \ref{thm:CMS-PropertyConvergence} and \ref{thm:CMS-PropertiesCauchy}; 
Propositions \ref{prop:CMS-IteratedContractionPrinciple} and \ref{prop:CMS-ExistenceFP}) give a method for obtaining convergence theorems (with error estimates) for Picard iteration 
and fixed point theorems in a cone metric space over a solid vector space.  

Finally, let us note that we have come to the idea of a general theory of cone metric spaces (over a solid vector spaces) dealing with convergence problems of some iterative methods for finding all zeros of a polynomial $f$ simultaneously (i.e., as a vector in $\Cset^n$, where $n$ is the degree of $f$).
In our next papers we will continue studying the cone metric space theory and its applications.
For instance, we shall show that almost all results given in Proinov \cite{Pro09,Pro10} can be extended in cone metric spaces over a solid vector space. Also we shall present new convergence theorems for some iterative methods for finding zeros of a polynomial simultaneously. These results generalize, improve and complement a lot of of results given in
the monographs of  Sendov, Andreev, Kyurkchiev \cite{SAK94} and Petkovic \cite{Pet08}.
In particular, it turns out that the cone norms in $\Cset^n$ give better a priori and a posteriori error estimates for iterative methods in $\Cset^n$ than usual norms.

\end{document}